\documentclass[a4paper, 11pt, leqno]{article}

\usepackage[T1]{fontenc}
\usepackage[utf8]{inputenc}
\usepackage{lmodern}
\usepackage{lmodern}

\usepackage[scale=0.83,centering]{geometry}
\usepackage{amsmath,amsfonts, mathabx,amssymb,amsthm,graphicx,bbm}
\usepackage{color,mathrsfs,tikz,hyperref}
\usepackage[font={footnotesize}]{caption}

\numberwithin{equation}{section}

\newcommand{\bbE}{{\ensuremath{\mathbbm E}} }

\newcommand{\bbN}{{\ensuremath{\mathbbm N}} }

\newcommand{\bbP}{{\ensuremath{\mathbbm P}} }

\newcommand{\bbR}{{\ensuremath{\mathbbm R}} }

\newcommand{\bbZ}{{\ensuremath{\mathbbm Z}} }

\newcommand{\cA}{{\ensuremath{\mathcal A}} }

\newcommand{\cD}{{\ensuremath{\mathcal D}} }

\newcommand{\cG}{{\ensuremath{\mathcal G}} }

\newcommand{\cM}{{\ensuremath{\mathcal M}} }
\newcommand{\cN}{{\ensuremath{\mathcal N}} }

\newcommand{\cP}{{\ensuremath{\mathcal P}} }

\newcommand{\cR}{{\ensuremath{\mathcal R}} }

\newcommand{\cT}{{\ensuremath{\mathcal T}} }

\newcommand{\cW}{{\ensuremath{\mathcal W}} }
\newcommand{\cX}{{\ensuremath{\mathcal X}} }

\newcommand{\bE}{{\ensuremath{\mathbf E}} }

\newcommand{\bM}{{\ensuremath{\mathbf M}} }

\newcommand{\bP}{{\ensuremath{\mathbf P}} }

\newcommand{\bY}{{\ensuremath{\mathbf Y}} }
\newcommand{\bZ}{{\ensuremath{\mathbf Z}} }

\newcommand{\ga}{\alpha}
\newcommand{\gb}{\beta}

\renewcommand{\epsilon}{\varepsilon}
\renewcommand{\theta}{\vartheta}

\renewcommand{\tilde}{\widetilde}

\newcommand{\ind}{\mathbbm{1}}
\newcommand{\dd}{{\ensuremath{\mathrm d}} }

\newcommand{\ent}{\mathrm{Ent}}
\newcommand{\hatent}{\hat{\mathrm{E}}\mathrm{nt}}

\newtheorem{theorem}{Theorem}[section]
\newtheorem{definition}{Definition}[section]
\newtheorem{lemma}[theorem]{Lemma}

\newtheorem{proposition}[theorem]{Proposition}
\newtheorem{remark}{Remark}[section]

\title{Non-directed polymers in random environments with range penalties: the high-dimensional case}

\author{
Niccol\`o~Torri
\footnote{MODAL'X, Universit\'{e} Paris Nanterre, 92000 Nanterre, France. Email: ntorri@parisnanterre.fr,
Orcid number : 0000-0002-4778-1305}
\and
Ran~Wei
\footnote{Department of Financial and Actuarial Mathematics, Xi'an Jiaotong-Liverpool University, 111 Ren'ai Road, Suzhou, Jiangsu Province, 215123, China.
Email: ran.wei@xjtlu.edu.cn}
}

\date{}

\begin{document}

\maketitle


\begin{abstract}
We study a \textit{non-directed polymer model} in random environments. The polymer is modeled by a simple symmetric random walk $S$ on $\bbZ^d$ with $d\geq2$, and the random environment is modeled by i.i.d.\ random variables whose tail probability decays polynomially. The interaction between the polymer and the random environment is captured by a Gibbs transform: at time $N$, the law of $S$ is tilted by the factor $\exp(\sum_{x\in\cR_N}(\beta\omega_x-h))$, where $\cR_N$ is the range of $S$ up to time $N$, $\beta\geq0$ is the inverse temperature, and $h\in\bbR$ is an external field. By appropriately tuning $\beta=\beta_N$ and $h=h_N$, we establish the phase diagram, analyze the fluctuations of $S$ under the Gibbs transform, and derive the scaling limits of the (logarithmic) partition function. This paper is a follow-up work of \cite{BTW21}. The main novelty and challenge arise from tuning the external field $h$, which brings in various \textit{range penalties}, unlike in \cite{BTW21}, where $h$ is fixed and serves only as a centering term for the random environment.

\vspace{0.2cm}
\textit{2010 Mathematics Subject Classification}:
82D60, 60K37, 60G70

\vspace{0.2cm}
\textit{Keywords}:
Random polymer, Random walk, Range, Heavy-tail distributions, Weak-coupling limit, Folding, Stretching
\end{abstract}



\section{Introduction}\label{sec:intro}
In this paper, we study a \textit{non-directed polymer model}, which describes an interaction between the polymer chain and the random environment (also referred to as \textit{disorder}). Specifically, each time the polymer chain visits a new site, it receives energy from the disorder and penalty from the external field. These competing effects jointly influence its configuration. We model the polymer by a simple symmetric random walk on $\bbZ^d$ and model the random environment by a family of i.i.d.\ random variables. A precise definition of the model will be provided in Section \ref{subsec:def_mode}.

The non-directed polymer model was originally introduced in \cite{H19}. It is closely related to two classical models: the \textit{random walk penalized by its range} (see \cite{DV1979}) and the \textit{directed polymer model} (see \cite{C17,Z24}). In the former, the law of the random walk is tilted solely by a deterministic external field, without the presence of a random environment. Our model may therefore be viewed as a disordered perturbation of that framework. In the directed polymer model, by contrast, the walk interacts with the environment at every time step---though, due to the directedness constraint, each time-space site is visited at most once. In our setting, however, the walk only interacts with the environment the first time it visits a new site (in the spirit of the excited random walk studied in \cite{BW03} or more generally the cookie random walk studied in \cite{Zer05}).

In the first study of the non-directed polymer model \cite{H19}, the weak coupling limits for the partition function were established under the assumption that the random environment admits finite exponential moments. More recently, the one-dimensional case has been analyzed in detail in \cite{BHTW20, B24}. For dimensions $d\geq2$, under heavy-tailed random environments and a \textit{fixed} external field $h$, the model was investigated in \cite{BTW21}, where the scaling limits for the (logarithmic) partition function were established. In this work, we extend the results of \cite{BTW21} by allowing the external field to vary with time, i.e., $h = h_N$, while still considering dimensions $d\geq2$ and heavy-tailed random environments.

\subsection{Definition of the model}\label{subsec:def_mode}
Let $S:=(S_n)_{n\geq0}$ be a simple symmetric random walk on $\bbZ^d$ starting at the origin, where $d\geq2$. The probability and the expectation of $S$ are denoted by $\bP$ and $\bE$ respectively. The \textit{range} of $S$ up to time $N$ is denoted by $\cR_N=\cR_N(S):=\{S_0, S_1, \cdots, S_N\}$, which is the set of distinct sites visited by the random walk up to time $N$.

Let $\omega:=(\omega_x)_{x\in\bbZ^d}$ be a family of i.i.d.\ random variables independent of $S$. The probability and the expectation of $\omega$ are denoted by $\bbP$ and $\bbE$ respectively. Throughout this paper, we assume that $\omega$ has a heavy-tailed distribution. Specifically, we assume that 
\begin{equation}\label{def:omega}
\bbP(\omega>t)\sim pt^{-\alpha}\quad\text{and}\quad\bbP(\omega<-t)\sim q t^{-\alpha},\quad \text{as\, } t\to +\infty,
\end{equation}
where $p,q>0$, and $\alpha\in(0,d)$ (for a discussion about the range of $\alpha$, see Section \ref{sec:heuristics}). Additionally, if $\ga>1$, we assume that $\bbE[\omega]=0$.

The \textit{polymer measure} at time $N$ is defined via a Gibbs transform
\begin{equation}\label{def:model}
\frac{\dd\bP_{N,\beta}^{\omega,h}}{\dd\bP}(S):=\frac{1}{\bZ_{N,\beta}^{\omega,h}}\exp\left(\sum\limits_{x\in\cR_N}(\beta\omega_x-h)\right)=\frac{1}{\bZ_{N,\beta}^{\omega,h}}\exp\left(\sum\limits_{x\in\cR_N}\beta\omega_x-h|\cR_N|\right),
\end{equation}
where $\beta\geq0$ is the inverse temperature, $h\geq0$ is an external field and
\begin{equation}\label{df:pf}
\bZ_{N,\beta}^{\omega,h}=\bE\left[\exp\left(\sum\limits_{x\in\cR_N}\beta\omega_x-h|\cR_N|\right)\right]
\end{equation}
is the \textit{partition function}, which ensures that $\bP_{N,\beta}^{\omega,h}$ is a probability measure.

Let us comment on the assumptions in \eqref{def:omega} and the condition $h\geq0$:

$\bullet$ The requirement $q>0$ is not essential, and our results are expected to remain valid even when $q=0$. Indeed, in random polymer models, polymer chains are intended to gain energy from the disorder. Hence, the right tail of $\omega$ plays a primary role, encouraging the polymer chain to explore. Assuming $q>0$ provides better control of the left tail of $\omega$. In particular, it simplifies the proof of Theorem \ref{thm:R5a} (see the penultimate paragraph in Section \ref{sec:R5}) and allows us to obtain the sharp asymptotics in \eqref{eq:lambda}. Note that if $q=0$, then $\omega\geq0$, and the exponent in \eqref{def:model} should be replaced by $\sum_{x\in\cR_N}(\beta(\omega_x-\bbE[\omega_x]\ind_{\{\alpha>1\}})-h)$ as in \cite{BTW21}. In that work, $\omega\geq0$ and the exponent was defined by $\sum_{x\in\cR_N}\beta(\omega_x-h)$ (note that $h$ is in the bracket, different from \eqref{def:model}), where $h$ was either negligible or chosen as $h=\bbE[\omega_x]\ind_{\{\alpha>1\}}$.

$\bullet$ As seen from \eqref{def:model}, environments with large positive energies are attractive to the random walk. At the same time, the external field $h\geq0$ penalizes trajectories with a long range, thereby encouraging the walk to fold. This leads to a competition between the attractive effect of random environments (we call it \textit{energy gain}) and the penalizing effect of the external field (we call it \textit{range penalty}). If $h<0$, then both effects encourage the random walk to explore, making the walk more likely be super-diffusive and resulting in a less rich phase diagram. We refer to \cite{BHTW20} for the case $h<0$ in dimension $d=1$.

\vspace{0.2cm}
We further allow both $\beta$ and $h$ to vary with $N$. Specifically, we set
\begin{equation}\label{def:betah}
\beta_N=\hat{\beta}N^{-\gamma},\quad h_N=\hat{h}N^{-\zeta}
\end{equation}
for constants $\hat{\beta},\hat{h}>0$ and $\gamma, \zeta\in\bbR$. To simplify notation, we write $\bP_N=\bP_{N,\beta_N}^{\omega, h_N}$, $\bE_N=\bE_{N,\beta_N}^{\omega,h_N}$ and $\bZ_N=\bZ_{N,\beta_N}^{\omega, h_N}$ when there is no ambiguity. Note that when $\gamma,\zeta>0$, this corresponds to the \textit{weak-coupling limit} (or the \textit{intermediate disorder region}) of the model, first studied in \cite{AKQ14} for the directed polymer model and later extended to various disordered systems in \cite{CSZ13}. Following \cite{BHTW20}, we also include the case $\gamma,\zeta<0$, in which the coupling strength increases with $N$.

In this paper, we investigate the scaling limits of $\log\bZ_N$ or $\bZ_N$ (under appropriate centering), as well as the \textit{end-to-end} exponent $\xi$ of $S$ under $\bP_N$. Informally, the end-to-end exponent is characterized by $$\bbE\bE_N^\omega[\max_{1\leq n\leq N}|S_n|]\approx N^\xi,$$ and a precise definition will be given in Definition \ref{def:xi}. These scaling limits and the exponent $\xi$ naturally depend on $d,\alpha,\gamma$ and $\zeta$. In Section \ref{sec:heuristics}, we provide a heuristic analysis to elucidate the interplay among these parameters and to establish the phase diagrams of the model. This heuristic analysis is pivotal to our paper.

\subsection{Main contributions} Let us highlight our contributions in this paper.

$\bullet$ First, we establish the complete phase diagrams for the non-directed polymer model for $d\geq2$ and $\alpha\in(0,\frac{d}{2})\cup(\frac{d}{2},d)$, extending those obtained in \cite{BTW21}. The behavior of the model differs significantly between the cases $\alpha \in (0, \frac{d}{2})$ and $\alpha \in (\frac{d}{2}, d)$. For $\alpha \in (\frac{d}{2}, d)$, we identify six regions in the phase diagram, labeled by $R_1$--$R_6$, whereas for $\alpha \in (0, \frac{d}{2})$, there are five regions, with the region $R_2$ absent, see Section \ref{sec:phasediagram}.

We characterize all phases for $\alpha\in(0,\frac{d}{2})$ and most phases for $\alpha\in(\frac{d}{2},d)$ by determining the scaling limits of $\log\bZ_N$ or $\bZ_N$. In particular, the region $R_1$ corresponds to the diffusive regime, and the regions $R_2$ and $R_3$ correspond to super-diffusive regimes (as studied in \cite{BTW21}), while the regions $R_4$, $R_5$, and $R_6$ are sub-diffusive regimes that do not appear in \cite{BTW21} (see Section~\ref{sec:main_result} for details).

$\bullet$ Second, we identify a new phase transition between the regions $R_4$ and $R_5$ when the range penalty and the entropy cost for sub-diffusive paths are balanced. This phenomenon does not occur in the one-dimensional disordered model studied in \cite{BHTW20,B24}, nor in high-dimensional homogeneous models with a fixed external field, as considered in \cite{DV1979,Bolt94,DFSX20,BC21}.

While the scaling limit of $\log\bZ_N$ in the region $R_5$ is consistent with earlier results in \cite{BHTW20,S90,DV1979}, the corresponding result in the region $R_4$ is new. Moreover, in the region $R_4$, we establish the sharp asymptotics $|\cR_N|\approx\bE[|\cR_N|]=O(N)$ under $\bP_N$ (see Theorem \ref{thm:R4}), revealing an interesting phenomenon: a suitably weak range penalty induces the random walk to fold, yet does not promote frequent self-intersections. In this regime, the entropy cost of self-intersections is too large to be compensated by a range penalty of insufficient strength.

$\bullet$ Third, we identify further phase transitions in the region $R_1$, which were not observable in \cite{BTW21} due to the lack of a varying external field therein. In particular, the centered and rescaled partition function may converge to a deterministic limit (see \eqref{eq:new_phase} and \eqref{eq:new}), in which the effect of the random environment does not appear in the first order fluctuation of $\log\bZ_N$, as is also the case in the regions $R_4$, $R_5$ and $R_6$. This naturally motivates the study of higher-order asymptotics and fluctuations of the partition function, as carried out in \cite{B24}.

$\bullet$ Finally, we develop a strategy to control the entropy cost associated with the range of $S$ when the range penalty is sufficiently strong (in particular, when $\zeta<\gamma-\frac{d-\alpha}{\alpha}$, see Figures \ref{diagram2} and \ref{diagram3}). This is achieved by comparing the entropy cost of a specific class of random walk configurations with a quantity derived from Hammersley's last-passage percolation (see Lemma \ref{lem:g>zda}, Theorem \ref{thm:transf-rangentr} and Proposition \ref{lem:Omega}; see also \cite{BT21}). These technical estimates are novel and may be of independent interest. 

\subsection{Organization of the paper}
The remainder of the paper is organized as follows. In Section~\ref{sec:main_result}, we state our main results, present the heuristics, and provide some discussion, conjectures, and relevant notation. In Section \ref{sec:R4}, we prove Theorem \ref{thm:R4}. In Section~\ref{sec:R6}, we prove Theorem~\ref{thm:R6} and establish a key technical result, Lemma~\ref{lem:g>zda}, which controls the energy gain associated with the range. In Section~\ref{sec:R5}, we prove Theorem~\ref{thm:R5a}, again relying on Lemma~\ref{lem:g>zda}. In Section \ref{ch:discuss}, we discuss the end-to-end exponent $\xi$ and the remaining part in the region $R_5$ for $\alpha\in(\frac{d}{2},d)$. Some technical estimates are collected in Appendix \ref{sec:appA}. Finally, we sketch the proofs for Theorems \ref{thm:R1a}-\ref{thm:A} in Appendix \ref{S-R123}, which are straightforward adaptations of the arguments in \cite{BTW21}.

\section{Main results}\label{sec:main_result}
In this section, we state our main results concerning the scaling limits of $\log \bZ_N$ or $\bZ_N$ (recall the definition of $\bZ_N$ below \eqref{def:betah}) and the end-to-end exponent $\xi$ in each regime. We then present the heuristics underlying the phase diagrams, which play a central role in the paper. We also provide further discussion and conjectures related to our results. Finally, we collect notation that will be used throughout the paper.

\subsection{Phase diagram}\label{sec:phasediagram}
In this paper, we restrict our attention to the case $\alpha<d$ (see Section \ref{sec:heuristics} for the reason). Below, we introduce the phase diagrams and describe the regions as functions of the parameters $\gamma, \zeta, \alpha$ and $d$. As will become clear in Section~\ref{sec:heuristics}, where the heuristics is presented, we need to consider the phase diagrams separately for the cases $\alpha \in \big(0, \frac{d}{2}\big)$ and $\alpha \in \big(\frac{d}{2}, d\big)$.

$\bullet$ For $\alpha\in(\frac d2,d)$, the phase diagram is composed by 6 regions:
\begin{enumerate} \itemsep -1pt
\item $R_1=\Big\{(\zeta, \gamma) \colon \zeta >1, \gamma >\frac{d}{2\alpha} \Big\}$, 
\item $R_2=\Big\{(\zeta, \gamma) \colon \zeta >0,  \frac{d-\alpha}{\alpha}<\gamma < \Big(\frac{2\alpha-d}{2\alpha}\zeta +\frac{d-\alpha}{\alpha}\Big)\wedge\frac{d}{2\alpha} \Big\}$,
\item $R_3=\Big\{(\zeta, \gamma) \colon \gamma < \zeta\wedge 0 +\frac{d-\alpha}{\alpha} \Big\}$,
\item $R_4=\Big\{(\zeta, \gamma) \colon \zeta\in (\frac2d,1), \gamma >\zeta\Big(1\wedge \frac{2\alpha-d}{2\alpha}\Big)+\frac{d-\alpha}{\alpha}\Big\}$,
\item $R_5=\Big\{(\zeta, \gamma) \colon \zeta\in (-1,\frac2d), \gamma >\zeta\Big(1\wedge \frac{2\alpha-d}{2\alpha}\Big)+\frac{d-\alpha}{\alpha}\Big\}$,
\item $R_6=\Big\{(\zeta, \gamma) \colon \zeta<-1, \gamma >\zeta+\frac{d-\alpha}{\alpha}\Big\}$.
\end{enumerate}
The regions $R_2$ and $R_3$ correspond to super-diffusive regimes, while $R_4$, $R_5$ and $R_6$ correspond to sub-diffusive regimes, and $R_1$ corresponds to the diffusive regime. Note that the region $R_4$ exists only for $d\ge3$. The corresponding phase diagram is shown in Figure \ref{diagram2}.
\begin{figure}[h!] 
\centering
\begin{tikzpicture}[scale=3.25]
\draw[->](-1.75, 0)--(1.5, 0) node[anchor=north] {$\zeta$};
\draw[->](0, -1.25)--(0, 1) node[anchor=east] {$\gamma$};
\draw[red, thick](0, 0.5)--(1, 0.75);
\draw[green, thick](1, 0.75)--(1.5, 0.75);
\draw[yellow, thick](1, 0.75)--(1, 1);
\draw (1, 0.85) node[anchor=west] {\small $(1,\frac{d}{2\ga})$};
\draw (1.45, 1) node[anchor=north] {$R_1$};
\draw[blue, thick](0, 0.5)--(1.5, 0.5);
\draw[teal, dashed](0.66, 0.75)--(0.25, 0.75);
\draw[teal, dashed](0.25, 0.75)--(0, 0.5);
\draw (-0.4, 0.7) node[anchor=west] {\small $\gamma=\frac{2\ga-d}{2\ga}\zeta+\frac{d-\ga}{\ga}$};
\draw (1, 0.675) node[anchor=north] {$R_2$};
\draw (0, 0.4) node[anchor=west] {\small $(0, \frac{d-\ga}{\ga})$};
\draw[brown, thick](0, 0.5)--(-1.75, -1.25);
\draw (-0.5, -0.275) node[anchor=north] {\small $\gamma=\zeta+\frac{d-\ga}{\ga}$};
\draw (0.7, -0.5) node[anchor=north] {$R_3$};
\draw[violet, thick](-1, -0.5)--(-1, 1); 
\draw[orange, thick](0.66,0.66)--(0.66,1);
\draw (0.66, 1.2) node[anchor=north] {$\zeta=\frac2d$};
\draw (0.8, 1) node[anchor=north] {$R_4$};
\draw (-1.3, -0.35) node[anchor=north] {\small $(-1, \frac{d}{\ga}-2)$};
\draw (-0.5, 0.7) node[anchor=north] {$R_5$};
\draw (-1.15, 0.2) node[anchor=north] {$R_6$};
\draw (-2.2, 1.2) node[anchor=west] {\small $R_1: \xi=\frac{1}{2}$};
\draw (-2.2, 1) node[anchor=west] {\small $R_2: \xi=\frac{\ga(1-\gamma)}{2\ga-d}$};
\draw (-2.2, 0.8) node[anchor=west] {\small $R_3: \xi=1$};
\draw (-2.2, 0.6) node[anchor=west] {\small $R_4: \xi=\frac{\zeta}{2}\in(\frac1d,\frac12)$};
\draw (-2.2, 0.4) node[anchor=west] {\small $R_5: \xi=\frac{1+\zeta}{d+2}\in(0,\frac1d)$};
\draw (-2.2, 0.2) node[anchor=west] {\small $R_6: \xi=0$};
\end{tikzpicture}
\caption{The phase diagram for $\alpha\in(\frac{d}{2},d)$ (illustrated for the representative case $d=3$, $\alpha=2$). The small region below the dashed line within $R_5$ remains unresolved.}
\label{diagram2}
\end{figure}

$\bullet$ For $\alpha\in(0,\frac d2)$, we have the following 5 regions (the region $R_2$ disappears):
\begin{enumerate}
\item $R_1=\Big\{(\zeta, \gamma) \colon \zeta >1, \gamma >\frac{d-\alpha}{\alpha} \Big\}$, 
\item $R_3=\Big\{(\zeta, \gamma) \colon \gamma < \zeta\wedge 0 +\frac{d-\alpha}{\alpha} \Big\}$,
\item $R_4=\Big\{(\zeta, \gamma) \colon \zeta\in (\frac2d,1), \gamma >\frac{d-\alpha}{\alpha}\Big\}$,
\item $R_5=\Big\{(\zeta, \gamma) \colon \zeta\in (-1,\frac2d), \gamma >\zeta\wedge 0+\frac{d-\alpha}{\alpha}\Big\}$,
\item $R_6=\Big\{(\zeta, \gamma) \colon \zeta<-1, \gamma >\zeta+\frac{d-\alpha}{\alpha}\Big\}$.
\end{enumerate}
The corresponding phase diagram is shown in Figure \ref{diagram3}.
\begin{figure}[h!] 
\centering
\begin{tikzpicture}[scale=3]
\draw[->](-1.75, 0)--(1.5, 0) node[anchor=north] {$\zeta$};
\draw[->](0, -0.3)--(0, 2) node[anchor=east] {$\gamma$};
\draw[yellow, thick](1, 1.2)--(1, 2);
\draw (1, 1.1) node[anchor=west] {\small $(1,\frac{d-\ga}{\ga})$};
\draw (1.3, 1.7) node[anchor=north] {$R_1$};
\draw[blue, thick](0, 1.2)--(1.5, 1.2);
\draw (0, 1.1) node[anchor=west] {\small $(0, \frac{d-\ga}{\ga})$};
\draw[brown, thick](0, 1.2)--(-1.5, -0.3);
\draw (-0.5, 0.45) node[anchor=north] {\small $\gamma=\zeta+\frac{d-\ga}{\ga}$};
\draw (0.25, 0.5) node[anchor=north] {$R_3$};
\draw[violet, thick](-1, 0.2)--(-1, 2); 
\draw[orange, thick](0.66,1.2)--(0.66,2);
\draw (0.5, 1.1) node[anchor=west] {\small $(\frac2d, \frac{d-\ga}{\ga})$};
\draw (0.8, 1.7) node[anchor=north] {$R_4$};
\draw (-1.3, 0.4) node[anchor=north] {\small $(-1, \frac{d}{\ga}-2)$};
\draw (-0.25, 1.7) node[anchor=north] {$R_5$};
\draw (-1.35, 1) node[anchor=north] {$R_6$};
\draw (-2.2, 2) node[anchor=west] {\small $R_1: \xi=\frac{1}{2}$};
\draw (-2.2, 1.8) node[anchor=west] {\small $R_3: \xi=1$};
\draw (-2.2, 1.6) node[anchor=west] {\small $R_4: \xi=\frac{\zeta}{2}\in(\frac1d,\frac12)$};
\draw (-2.2, 1.4) node[anchor=west] {\small $R_5: \xi=\frac{1+\zeta}{d+2}\in(0,\frac1d)$};
\draw (-2.2, 1.2) node[anchor=west] {\small $R_6: \xi=0$};
\end{tikzpicture}
\caption{The phase diagram for $\alpha\in(0,\frac{d}{2})$ (illustrated for the representative case $d=3, \ga=\frac54$).}
\label{diagram3}
\end{figure}

\vspace{0.2cm}
Before stating our main results, we first provide a rigorous definition of the end-to-end exponent $\xi$.
\begin{definition}\label{def:xi}
We say that $(S_n)_{0\leq n\leq N}$ has end-to-end fluctuations $r_N$ under $\bP_N$, if for any $\epsilon>0$, there exists some $\eta\in(0,1)$, such that for $N$ large enough,
\begin{equation*}
\bbP\Big(\bP_N\Big(\max_{0\le n\le N}|S_n|\in[\eta,\eta^{-1}]r_N\Big)>1-\epsilon\Big)>1-\epsilon.
\end{equation*}
If $r_N$ is of the form $N^\xi$, then $\xi$ is the end-to-end exponent.
\end{definition}
We then study the exponent $\xi$ and the non-trivial scaling limits of $\log\bZ_N$ or $\bZ_N$  according to the phase diagrams in Figures \ref{diagram2} and \ref{diagram3}. To avoid overly technical complications, we exclude all boundary cases and restrict our analysis to the interior of each region.

Our results for the regions $R_4$, $R_5$, and $R_6$ are new. In these regions, the energy gain is negligible, and the random walk is sub-diffusive. Inspired by the ideas in \cite{BT19a, BHTW20, BTW21}, we perform a delicate comparison between the energy gain and the range penalty, and show that $\bZ_N$ is comparable to $\bZ_{N,\beta_N=0}^{\omega,h_N}$, where the latter is the partition function of the corresponding homogeneous model. It is known from \cite{DV1979} that
\begin{equation*}
\log\bZ_{N,\beta_N=0}^{\omega,h_N}\approx\max_{(S_n)_{0\leq n\leq N}}\{\text{range penalty}-\text{entropy cost}\}.
\end{equation*}
Hence, the appropriate scaling factor for $\log \bZ_N$ is $N^{-(1-2\xi)}$ for $\xi\in(0, \frac{1}{2})$ (see the discussion of the sub-diffusive entropy cost in Section~\ref{sec:heuristics}), and the scaling limit is given by a limiting variational problem, which admits an explicit solution in \cite{DV1979} (at least for $\xi\in\big(0,\frac{1}{d}\big)$). We will state the results for the regions $R_4$, $R_5$, and $R_6$ in Section~\ref{sec:results}.

The regions $R_1$, $R_2$, and $R_3$ were studied in depth in \cite{BTW21}, where the external field $h$ is fixed and serves as a centering term for $\omega$. In our setting, a straightforward comparison shows that the external field $h_N$ is negligible in the regions $R_2$ and $R_3$, and thus the results from \cite{BTW21} apply directly. In contrast, the region $R_1$ exhibits a richer phase diagram. In particular, we identify an additional phase transition in $R_1$, depending on the influence of the energy gain and the range penalty on higher-order terms of $\bZ_N$. The proof proceeds by adapting the corresponding arguments in \cite{BTW21}. For completeness, we restate the results for the regions $R_1$, $R_2$, and $R_3$ in Section~\ref{sec:results+}, along with a sketch of the proofs in Appendix~\ref{S-R123}.

\subsection{Main results in the regions $R_4, R_5$ and $R_6$}\label{sec:results}
We now state our results for the regions $R_4$, $R_5$ and $R_6$.

In the region $R_4$, the random walk tends to fold, while still having enough space to extend. This region exists only when the dimension satisfies $d \geq 3$. It is well known that for $d \geq 3$ (see \cite{DE51}),
\begin{equation}\label{eq:range_limit}
\gamma_d:=\lim_{N\to\infty}\frac{1}{N}\bE[|\cR_N|]>0.
\end{equation}
We then have the following result.
\begin{theorem}\label{thm:R4}
Let $\alpha\in(0,d)$ and $(\zeta,\gamma)\in R_4$. In particular, this requires $d\geq3$. Then for $\xi=\frac{\zeta}{2}\in(\frac1d,\frac12)$, we have that
\begin{equation*}
\lim\limits_{N\to\infty}\frac{1}{N^{1-2\xi}}\log\bZ_N=-\hat{h}\gamma_d\quad\text{in}~\bbP\text{-probability},
\end{equation*}
where $\gamma_d>0$ is given in \eqref{eq:range_limit}, and for any $\delta>0$, $\bP_N\big(\big|\frac{|\cR_N|}{N}-\gamma_d\big|<\delta\big)\overset{N\to\infty}{\longrightarrow}1$ in $\bbP$-probability.
\end{theorem}

In the region $R_5$, the random walk tends to fold and is expected to fill a sub-diffusive ball (see \cite{Bolt94, DFSX20, BC21}). We state our results by considering three cases separately, as the analysis differs across these cases.
\begin{theorem}\label{thm:R5a}
If any one of the following conditions holds:
\begin{itemize}
\item[(1)] $\alpha\in(0,d)$ and $(\zeta,\gamma)\in R_5$ with $\gamma>\zeta+\frac{d-\alpha}{\alpha}$,
\item[(2)] $\alpha\in(0,\frac{d}{2})$ and $(\zeta,\gamma)\in R_5$ with $\gamma>\frac{d-\alpha}{\alpha}$,
\item[(3)] $\alpha\in[\frac{d}{2},d)$ and $(\zeta,\gamma)\in R_5$ with $\gamma\geq\frac{d}{2\alpha}$,
\end{itemize}
then for $\xi=\frac{1+\zeta}{2+d}$, we have the following convergence in $\bbP$-probability
\begin{equation}\label{eq:R5scaling}
\lim\limits_{N\to\infty}\frac{1}{N^{1-2\xi}}\log\bZ_N=\lim\limits_{N\to\infty}\frac{1}{N^{1-2\xi}}\log\bZ_{N,\beta_N=0}^{\omega,h_N}=-c_d(\hat{h}):=-\hat{h}^{\frac2{d+2}}\Big(\frac{d+2}{2}\Big)\Big(\frac{2\lambda_1}{d}\Big)^{\frac{d}{d+2}},
\end{equation}
where $\lambda_1$ is the principal Dirichlet eigenvalue of $-\frac{1}{2d}\Delta$ in the ball with unit volume in $\bbR^d$ centered at the origin, and the explicit expression for $c_d(\hat{h})$ was derived in \cite{DV1979}.
\end{theorem}

\begin{remark}
Taking the union of cases (1)-(3) in Theorem \ref{thm:R5a}, we observe that for $\alpha\in(0,\frac{d}{2})$, (see Figure \ref{diagram3}), the scaling limit of $\log Z_N$ has been identified throughout the entire region $R_5$. In contrast, for $\alpha\in(\frac{d}{2}, d)$, (see Figure \ref{diagram2}), there remains a small sub-region of $R_5$ where the problem is unresolved, where $0<\zeta<\frac{2}{d}$ and $\frac{2\alpha-d}{2\alpha}\zeta+\frac{d-\alpha}{\alpha}<\gamma\leq(\zeta+\frac{d-\alpha}{\alpha})\wedge\frac{d}{2\alpha}$. The techniques in this small part could be very sophisticated and we will discuss some possible approaches at the end of this paper in Section \ref{ch:discuss}.
\end{remark}

In the region $R_6$, the range penalty is so strong that the random walk is confined to two points, and we have the following result.
\begin{theorem}\label{thm:R6}
Let $\alpha \in (0,d)$ and $(\zeta,\gamma)\in R_6$. Then
\begin{equation}
\lim_{N\to +\infty} N^{\zeta}\log \bZ_N=
\lim_{N\to +\infty} N^{\zeta}\log \bZ_N(|\cR_N|=2)=-2\hat h\quad\text{in}~\bbP\text{-probability}
\end{equation}
and $\bP_N(|\cR_N|=2)\overset{N\to\infty}{\longrightarrow}1$ in $\bbP$-probability.
\end{theorem}
\subsection{Extended results in the regions $R_1, R_2$ and $R_3$}\label{sec:results+}
In this section, we extend the results in \cite{BTW21} to our setting, corresponding to the regions $R_1, R_2$ and $R_3$.

The region $R_1$ is the diffusive regime, in which the scaling limits of $\log \bZ_N$ or $\bZ_N$ depend crucially on the relationship between $\alpha$ and $d$. We state these results separately in Theorems~\ref{thm:R1a}--\ref{thm:R1c}.

\begin{theorem}\label{thm:R1a}
Let $\ga\in(2\vee\frac{d}{2},d)$ (in particular $d\geq 3$) and $(\zeta,\gamma)\in R_1$ (see Figure \ref{diagram2}). Then $(S_n)_{0\leq n \leq N}$ has end-to-end fluctuations of order $N^{\xi}$ under $\bP_N$ with $\xi=\frac12$. Without loss of generality, we assume that ${\rm Var_\mathbb{P}}(\omega)=1$. Take $a_N=N^{1/4}$ for $d=3$, $a_N=(\log N)^{1/2}$ for $d=4$, and $a_N=1$ for $d\geq 5$. Then as $N\to\infty$,
\begin{equation}\label{eq:a>2}
\frac{1}{a_N\beta_N}\left(\bZ_N-\bE\left[e^{(\frac{1}{2}\beta_N^2-h_N)|\cR_N|}\right]\right)\overset{(d)}{\longrightarrow}\begin{cases}
\cN(0,\sigma_d^2),\quad&\text{for}~d=3,4,\\
\cX,\quad&\text{for}~d\geq5,
\end{cases}
\end{equation}
where $\cN(0,\sigma_d^2)$ is a centered Gaussian distribution with an explicit variance $\sigma_d^2$ and
\begin{equation}\label{def:cX}
\cX:=\sum\limits_{x\in\bbZ^d}\omega_x\bP(x\in\cR_\infty)=\bE\left[\sum\limits_{x\in\cR_\infty}\omega_x\right],
\end{equation}
which is a well-defined random variable.
\end{theorem}
The above theorem is a counterpart to equation (2.7) in \cite[Theorem~2.12]{BTW21}. The main modification is that, in \eqref{eq:a>2}, we adopt a different centering for $\bZ_N$ due to the presence of the external field $h_N$.

\begin{theorem}\label{thm:R1b}
Let $\ga\in(\frac{d}{2},2)$ (in particular $d=2,3$) and $(\zeta,\gamma)\in R_1$ (see Figure \ref{diagram2}). Then $(S_n)_{0\leq n \leq N}$ has end-to-end fluctuations of order $N^{\xi}$ under $\bP_N$ with $\xi=\frac12$ and we have the following scaling limits for $\bZ_N$.

\noindent{\rm (i)} If $\zeta\in(1,\gamma+\frac{d(\alpha-1)}{2\alpha})$, then
\begin{equation}\label{eq:new_phase}
\frac{1}{h_N\bE[|\cR_N|]}\log\bZ_N\overset{\bbP}{\longrightarrow}-1.
\end{equation}

\noindent{\rm (ii)} If $\zeta>\gamma+\frac{d(\alpha-1)}{2\alpha}$, then setting $v_N=\log N$ for $d=2$ and $v_N=\sqrt{N}$ for $d=3$, we have that
\begin{equation}\label{eq:old_phase}
\frac{v_N}{\beta_N N^{\frac{d}{2\ga}}}\log\bZ_N\overset{(d)}{\longrightarrow}\cW,
\end{equation}
where
\begin{equation}\label{def:cW}
\cW:=\int_{\bbR^d\times\bbR}\omega f(x)(\cP-\eta)(\dd x,\dd w)
\end{equation}
is a well-defined random variable with
\begin{equation}
f(x):=\begin{cases}
\int_{\|x\|^2/2}^{\infty}u^{-1}e^{-u}\dd u,\quad&\text{if}~d=2,\\
2\lambda_d\int_0^1\rho_d(u,x)\dd u,\quad&\text{if}~d=3,
\end{cases}
\end{equation}
a Poisson point process $\cP$ with intensity measure $\eta(\dd x,\dd\omega)=\alpha p^{\ind_{\{\omega>0\}}}q^{\ind_{ \{\omega<0\}}}|\omega|^{-(1+\alpha)}\dd x\dd \omega$, $\lambda_d=\bP(S_n\neq0,\forall n\geq1)$, and the Gaussian kernel $\rho_d(t,x)=(2\pi t/d)^{-d/2}\exp(-d\|x\|^2/2t)$.
\end{theorem}
The above theorem is a counterpart to equation (2.8) in \cite[Theorem 2.12]{BTW21}. Note that we find a new phase transition along the critical line $\zeta=\gamma+\frac{d(\alpha-1)}{2\alpha}$. The phase described in \eqref{eq:new_phase} is new, where the range penalty dominates the energy gain in the higher order terms of $\bZ_N$.

\begin{theorem}\label{thm:R1c}
Let $\ga\in(0,\frac{d}{2})$ and $(\zeta,\gamma)\in R_1$ (see Figure \ref{diagram3}). Then $(S_n)_{0\leq n \leq N}$ has end-to-end fluctuations of order $N^{\xi}$ under $\bP_N$ with $\xi=\frac12$ and we have the following scaling limits of $\bZ_N$.

\noindent{\rm (i)} If $\ga\in(\frac{d}{d-2},\frac{d}{2})$ (in particular $d\geq5$, $\ga>1$), then
\begin{equation}\label{eq:old1}
\frac{1}{\beta_N}\left(\bZ_N-\bE\left[e^{-h_N|\cR_N|}\right]\right)\overset{(d)}{\longrightarrow}\cX,
\end{equation}
where $\cX$ has been defined in \eqref{def:cX}.

\noindent{\rm (ii)} If $\ga<(\frac{d}{d-2}\wedge\frac{d}{2})$ (in particular $\alpha<2$) and $\alpha\neq1$, then we have the following.

{\rm (iia)} If $\zeta\in(1,\gamma+\frac{d(\ga-1)}{2\ga})$, then we have that
\begin{equation}\label{eq:new}
\frac{1}{h_N\bE[|\cR_N|]}\log\bZ_N\overset{\bbP}{\longrightarrow}-1.
\end{equation}

{\rm (iib)} If $\zeta>\gamma+\frac{d(\ga-1)}{2\ga}$, then setting $v_N=\log N$ for $d=2$ and $v_N=N^{\frac{d}{2}-1}$ for $d\geq3$, we have that
\begin{equation}\label{eq:old2}
\frac{v_N}{\beta_N N^{\frac{d}{2\ga}}}\log\bZ_N\overset{(d)}{\longrightarrow}\cW,
\end{equation}
where $\cW$ has been defined in \eqref{def:cW}.
\end{theorem}
The above theorem is a counterpart to \cite[Theorem 2.15]{BTW21}. As in Theorem \ref{thm:R1a}, the centering for $\bZ_N$ in \eqref{eq:old1} is modified due to the presence of the external field $h_N$. We also note that, in the case $\alpha\in(0,\frac d2)$, the statement of our results is much simpler than that in \cite[Theorem 2.15]{BTW21}. Since we restrict attention to the interior of the region $R_1$, where $\lim_{N\to\infty}N^{(d-\alpha)/\alpha}\beta_N=0$, the key quantity $\hat{\cT}_\beta$ in \cite[Theorem 2.2, Proposition 2.5, Theorem 2.15]{BTW21} satisfies $\hat{\cT}_\beta=\hat{\cT}_0=0$, $\bbP$-a.s.. Consequently, there is no need to emphasize the conditioning on the event $\{\hat{\cT}_\beta=0\}$ as in \cite[Theorem 2.15]{BTW21}.

\begin{remark}
It is natural to ask whether, in \eqref{eq:a>2} and \eqref{eq:old1}, one can obtain convergence to a deterministic limit by replacing the scaling factor with $h_N \bE[|\cR_N|]$, as in \eqref{eq:new_phase} and \eqref{eq:new}. Although this problem is tractable, we choose not to pursue it in detail for the following two reasons.

First, such an approach would require a comparison between $h_N \bE[|\cR_N|]$ and $a_N \beta_N$ in \eqref{eq:a>2}, analogous to the comparison between $h_N \bE[|\cR_N|]$ and $\beta_N N^{d/2\alpha}$ carried out in \eqref{eq:new_phase} and \eqref{eq:new} (see Appendix~\ref{S-R123} for details). However, the factor $a_N$ has different expressions depending on the dimension $d$, and a case-by-case analysis would lead to rather lengthy statements of the results.

Second, to obtain a deterministic limit with scaling factor $h_N \bE[|\cR_N|]$, one would need to introduce different centering terms that do not depend on $h_N$. This would result in an unnatural comparison with the existing formulations in \eqref{eq:a>2} and \eqref{eq:old1}.

Given that we have already identified additional phase transitions in Theorem~\ref{thm:R1b} and case~(ii) of Theorem~\ref{thm:R1c}, we consider the formulations in \eqref{eq:a>2} and \eqref{eq:old1} to be sufficient.
\end{remark}

The region $R_2$ corresponds to a super-diffusive regime, where the energy gain and the entropy cost are balanced. We have the following result.
\begin{theorem}\label{thm:B}
Let $\ga\in (\frac{d}{2},d)$ and $(\zeta,\gamma) \in R_2$ (see Figure \ref{diagram2}).
Then $(S_n)_{0\leq n \leq N}$ has end-to-end fluctuations of order $N^{\xi}$ under $\bP_N$, where $\xi= \frac{\alpha (1-\gamma)}{2\ga-d} \in (\frac12,1)$.
Additionally, as $N\to\infty$, we have the following convergence in distribution to a well-defined random variable,
\begin{equation}
\label{eq:VarPr}
\frac{1}{N^{2\xi -1}}\log \bZ_N \overset{(d)}{\longrightarrow}\cT_{\hat{\beta}} :=\sup_{s\in \cD, \ent(s) <+\infty} 
\big\{ \hat{ \gb} \pi(s) -  \ent(s) \big\}.
\end{equation}
Here $\pi(s):=\sum_{(x,\omega)\in\cP}\omega\mathbbm{1}_{\{x\in s[0,1]\}}$, where $\cP$ is the  Poisson point process introduced in Theorem \ref{thm:R1b} and
\begin{equation}
\cD:=\{s:[0,1]\to\bbR^d: s(0)=0.~s~\text{is continuous and a.e. differentiable}\},
\end{equation}
and
\begin{equation}
\ent(s):=\frac d2\Big(\int_0^1\|s'(t)\|\dd t\Big)^2.
\end{equation}
\end{theorem}
The above theorem coincides with \cite[Theorem~2.7]{BTW21}.

The region $R_3$ corresponds to a super-diffusive regime, where the energy gain dominates. We have the following result.
\begin{theorem}\label{thm:A}
Let $\ga\in (0,d)$ and $(\zeta,\gamma)\in R_3$ (see Figures \ref{diagram2} and \ref{diagram3}). Then $(S_n)_{0\leq n \leq N}$ has end-to-end fluctuations of order $N$ under $\bP_N$, and as $N\to\infty$, we have the following convergence in distribution to a well-defined random variable, 
\begin{equation}\label{def:varprob1}
\frac{1}{ \gb_N N^{d/\alpha}} \log \bZ_N  \stackrel{(d)}{ \longrightarrow} \hat\cT_\infty:= \sup_{s\in \cD, \hatent(s) <+\infty}\pi(s).
\end{equation}
Here
\begin{equation}
\hatent(s):=\inf_{\varphi\in\Phi}\int_0^1 J_d((s\circ\varphi)'(t))\dd t,
\end{equation}
where $J_d(\cdot)$ is a rate function obtained by
\begin{equation}
J_d(x):=-\lim\limits_{N\to\infty}\frac{1}{N}\log\bP(S_N=\lfloor xN\rfloor)
\end{equation}
and
\begin{equation}
\Phi:=\{\varphi:[0,1]\overset{\rm onto}{\longrightarrow}[0,1]: \varphi~\text{is non-decreasing and a.e. differentiable}\}.
\end{equation}
\end{theorem}
The above theorem is a special version of \cite[Theorem 2.2, Proposition 2.4]{BTW21}. Since we restrict attention to the interior of the region $R_3$, where $\lim_{N\to\infty}N^{(d-\alpha)/\alpha}\beta_N=\infty$, the condition $\hat{\cT}_{\infty}>0$ in \cite[Proposition 2.4]{BTW21} always holds.

\subsection{Heuristics}\label{sec:heuristics}
In this section, we discuss how the phase diagrams in Section~\ref{sec:phasediagram} are heuristically derived. We restrict our attention to the case $\alpha<d$, where $\alpha$ is the exponent in \eqref{def:omega}. The reason is that the effects of the random environment differ substantially between the cases $\alpha<d$ and $\alpha>d$, and therefore require different approaches.

When $\alpha < d$, the tail of $\omega$ is sufficiently heavy that the main contribution to $\sum_{x \in \cR_N} \omega_x$ is dominated by a few exceptionally large random variables. In contrast, when $\alpha > d$, the lighter tail implies that the sum is determined collectively by many moderate-valued random variables. The case $\alpha > d$ presents significant technical challenges and remains largely open (see also \cite[Section~2.5.4]{BTW21}). 

We now provide a heuristic analysis of the scales of the energy gain from the random environment, the range penalty induced by the external field, and the entropy cost associated with super-diffusive or sub-diffusive paths. We also note that here we exclude the case $\alpha = \frac{d}{2}$, since it is critical (see the phase transition between Figures~\ref{diagram2} and \ref{diagram3}), and a more delicate analysis is required. We will specify in the proofs (see Sections~\ref{sec:R4}–\ref{sec:R5}) when the case $\alpha=\frac{d}{2}$ is included. Recall that $\xi$ is the end-to-end exponent.
\begin{itemize} \itemsep -1pt
\item[(1)] For $\alpha<d$, by the heavy-tailed property, we have that $\sum_{x\in\cR_N}\omega_x\approx N^{d\xi/\alpha}$. Hence, the energy gain $\sum_{x\in\cR_N}\beta_N\omega_x$ is of order $N^{d\xi/\alpha-\gamma}$. 
\item[(2)] The scale of the range penalty depends on whether $\xi<\frac1d$ or $\xi>\frac1d$.
\begin{itemize} \itemsep -1pt
\item[(2a)] If $\xi<\frac1d\leq\frac12$, that is, the random walk folds itself in a sub-diffusive ball and fills it (see \cite{BC21,DFSX20}), then $|\cR_N|\approx N^{d\xi}$ and the range penalty is of order $N^{d\xi-\zeta}$.
\item[(2b)] If $\xi>\frac1d$, then the walk has enough space to stretch regardless of folding. In this case, we expect that $|\cR_N|\approx N$ and the range penalty is of order $N^{1-\zeta}$.
\end{itemize}
\item[(3)] Finally, by classical large deviation estimates, the entropy cost is of order
\begin{itemize} \itemsep -1pt
\item[(3a)] $\exp(-N^{1-2\xi})$ in the sub-diffusive regime (i.e., $\xi\in[0,\frac12)$);
\item[(3b)] $\exp(-N^{2\xi-1})$ in the super-diffusive regime (i.e., $\xi\in(\frac12,1]$).
\end{itemize} 
\end{itemize}

We now analyze the behavior of the polymer and establish the phase diagram through the competition among the energy gain, the range penalty, and the entropy cost.

\vspace{0.2cm}
\textbf{Case I} (\textit{entropy cost dominates}) In this case, the random walk is diffusive, so $\xi=\frac12\geq\frac1d$. There is no additional entropy cost, and both the energy gain and the range penalty are negligible. In particular, this requires that the exponents of the energy gain and the range penalty be negative, namely,
\begin{equation}\label{eq:R1}
\frac{d}{2\alpha}-\gamma<0\implies\gamma>\frac{d}{2\alpha}\quad\text{and}\quad1-\zeta<0\implies\zeta>1.
\end{equation}

We refer to this region as $R_1$ in the phase diagrams.

\vspace{0.2cm}
\textbf{Case II} (\textit{entropy cost-energy gain balance}) In this case, the random environment encourages the polymer to explore, leading to a super-diffusive behavior with $\xi\in(\frac{1}{2}, 1)$. By comparing the exponents of the energy gain, the range penalty, and the entropy cost, we obtain
\begin{equation}\label{eq:R2a}
\frac{d\xi}{\ga}-\gamma=2\xi-1\implies\xi=\frac{\ga(1-\gamma)}{2\ga-d},
\end{equation}
and
\begin{equation}\label{eq:R2b}
\frac{d\xi}{\ga}-\gamma=2\xi-1>1-\zeta\implies\gamma<\frac{2\ga-d}{2\ga}\zeta+\frac{d-\ga}{\ga}.
\end{equation}
At the boundary points $\xi=\frac{1}{2}$ and $\xi=1$, we have that
\begin{equation}\label{eq:R2c}
\xi=\frac{1}{2}\implies\gamma=\frac{d}{2\ga},\zeta>1\quad\text{and}\quad\xi=1\implies\gamma=\frac{d-\ga}{\ga},\zeta>0.
\end{equation}

We refer to this region as $R_2$ in the phase diagrams.

\vspace{0.2cm}
\textbf{Case III} (\textit{energy gain dominates}) In this case, we have $\xi=1$, and the boundary $\gamma=\frac{d-\alpha}{\alpha}$, $\zeta>0$ remains the same as in Case II (see \eqref{eq:R2c}). However, the conditions \eqref{eq:R2a} and \eqref{eq:R2b} are replaced by
\begin{equation}\label{eq:R3}
\frac{d}{\alpha}-\gamma>1\implies\gamma<\frac{d-\alpha}{\alpha}\quad\text{and}\quad\frac{d}{\alpha}-\gamma>1-\zeta\implies\gamma<\zeta+\frac{d-\alpha}{\alpha},\quad\text{respectively}.
\end{equation}

We refer to this region as $R_3$ in the phase diagrams.

\vspace{0.2cm}
\textbf{Case IV} (\textit{entropy cost-range penalty balance}) In this case, the range penalty forces the random walk to fold, leading to a sub-diffusive behavior with $\xi\in(0,\frac12)$. Since $0<\frac{1}{d}\leq\frac{1}{2}$, we distinguish between two sub-cases, namely $\xi\in(0, \frac{1}{d})$ and $\xi\in(\frac{1}{d},\frac{1}{2})$.

For $\xi\in(\frac1d,\frac12)$, which occurs only when the dimension $d\geq3$, a comparison of the exponents yields
\begin{equation}\label{eq:R4a}
1-\zeta=1-2\xi\implies\xi=\frac{\zeta}{2},
\end{equation}
and
\begin{equation}\label{eq:R4b}
1-\zeta=1-2\xi>\frac{d\xi}{\ga}-\gamma\implies\gamma>\frac{d+2\ga}{\ga}\xi-1=\frac{d+2\ga}{2\ga}\zeta-1.
\end{equation}
At the boundary points $\xi=\frac1d$ and $\xi=\frac12$, we have that
\begin{equation}\label{eq:R4c}
\xi=\frac{1}{d}\implies\zeta=\frac{2}{d}, \gamma>\frac{d+2\ga-d\ga}{d\ga}\quad\text{and}\quad\xi=\frac{1}{2}\implies\zeta=1,\gamma>\frac{d}{2\ga}.
\end{equation}

We refer to this region as $R_4$ in the phase diagrams.

For $\xi\in(0,\frac1d)$, similarly, we have that
\begin{equation}\label{eq:R5a}
d\xi-\zeta=1-2\xi\implies\xi=\frac{1+\zeta}{d+2},
\end{equation}
and
\begin{equation}\label{eq:R5b}
d\xi-\zeta=1-2\xi>\frac{d\xi}{\ga}-\gamma\implies\gamma>\frac{(d+2\ga)(1+\zeta)}{\ga(d+2)}-1=\frac{d+2\ga}{\ga(d+2)}\zeta+\frac{d(1-\ga)}{\ga(d+2)}.
\end{equation}
At the boundary points $\xi=0$ and $\xi=\frac1d$, we have that
\begin{equation}\label{eq:R5c}
\xi=0\implies\zeta=-1,\gamma>\zeta\quad\text{and}\quad\xi=\frac{1}{d}\implies\zeta=\frac{2}{d},\gamma>\frac{d+2\ga-d\ga}{d\ga}.
\end{equation}

We refer to this region as $R_5$ in the phase diagrams.

\vspace{0.2cm}
\textbf{Case V} (\textit{range penalty dominates}) In this case, we have $\xi=0$, and the boundary $\zeta=-1,\gamma>\zeta$ remains the same as in the region $R_5$ (see \eqref{eq:R5c}). However, the conditions \eqref{eq:R5a} and \eqref{eq:R5b} are replaced by
\begin{equation}\label{eq:R6+}
\zeta<-1\quad\text{and}\quad\gamma>\zeta,\quad\text{respectively}.
\end{equation}

We refer to this region as $R_6$ in the phase diagrams.

\vspace{0.2cm}
\textbf{Case VI} (\textit{energy gain-range penalty balance}) In this case, the entropy cost is negligible, so $\xi\in(0,1)$. Similar to Case IV, we distinguish among three sub-cases, namely $\xi\in(0,\frac1d)$, $\xi\in(\frac1d,\frac12)$ and $\xi\in(\frac12,1)$.

For $\xi\in(0,\frac1d)$, we have that
\begin{equation*}
\frac{d\xi}{\ga}-\gamma=d\xi-\zeta\implies\xi=\frac{\ga(\zeta-\gamma)}{d(\ga-1)},
\end{equation*}
which implies
\begin{equation*}
\begin{cases}
\zeta+\frac{1-\ga}{\ga}<\gamma<\zeta,\quad&\text{if}~\ga>1,\\
\zeta+\frac{1-\ga}{\ga}>\gamma>\zeta,\quad&\text{if}~\ga<1.
\end{cases}
\end{equation*}
A comparison of the exponents yields
\begin{equation*}
\frac{d\xi}{\alpha}-\gamma=d\xi-\zeta>1-2\xi\implies\begin{cases}
\gamma<\frac{2\alpha+d}{\alpha(d+2)}\zeta+\frac{d(1-\alpha)}{\alpha(d+2)},\quad&\text{if}~\alpha>1,\\
\gamma>\frac{2\alpha+d}{\alpha(d+2)}\zeta+\frac{d(1-\alpha)}{\alpha(d+2)},\quad&\text{if}~\alpha<1.
\end{cases}
\end{equation*}

For $\xi\in(\frac1d,\frac12)$, we have that
\begin{equation}\label{eq:energy-penalty1}
\frac{d\xi}{\ga}-\gamma=1-\zeta\implies\xi=\frac{\ga}{d}(1+\gamma-\zeta),
\end{equation}
which implies
\begin{equation}\label{eq:energy-penalty2}
\zeta+\frac{1-\ga}{\ga}\leq\gamma\leq\zeta+\frac{d-\ga}{\ga}.
\end{equation}
A comparison of the exponents yields
\begin{equation*}
\frac{d\xi}{\ga}-\gamma=1-\zeta>1-2\xi\implies\begin{cases}
\zeta+\frac{1-\ga}{\ga}<\gamma<\zeta+\frac{d-2\ga}{2\ga},\\
\gamma>\frac{d+2\ga}{2\ga}\zeta-1.
\end{cases}
\end{equation*}

Finally, for $\xi\in(\frac12,1)$, we still have \eqref{eq:energy-penalty1} and \eqref{eq:energy-penalty2}. A comparison of the exponents yields
\begin{equation*}
\frac{d\xi}{\ga}-\gamma=1-\zeta>2\xi-1\implies\begin{cases}
\zeta+\frac{d-2\ga}{2\ga}<\gamma<\zeta+\frac{d-\ga}{\ga},\\
\gamma<\frac{2\ga-d}{2\ga}\zeta+\frac{d-\ga}{\ga}.
\end{cases}
\end{equation*}

However, as will be suggested by the analysis below, the regime corresponding to the energy gain–range penalty balance does not occur, as it is dominated by other regimes. Therefore, we do not assign a separate label (such as $R_n$) to this case.

\vspace{0.2cm}
The regions identified above exhibit some overlap. Evidence from existing studies (in particular, \cite[Comment 2(c)]{BHTW20}) suggests that, in overlapping regions, those in which the range penalty is non-negligible are always dominated, due to the well-posedness of the limiting variational problem for the logarithmic partition function. This observation is also supported by our main results (see Sections~\ref{sec:results} and \ref{sec:results+}). We now examine the overlapping cases in light of this principle.

$\bullet$ By a straightforward (albeit somewhat tedious) comparison, we observe that the boundaries $\gamma = \frac{2\alpha - d}{2\alpha}\zeta + \frac{d - \alpha}{\alpha}$ of the region $R_2$ and $\gamma = \zeta + \frac{d - \alpha}{\alpha}$ of the region $R_3$ are the top-leftmost among all candidate boundaries arising from overlapping regions. Consequently, these boundaries confine the regions in which the range penalty is non-negligible to their minimal extent. As a result, the regime corresponding to the energy gain–range penalty balance also disappears.

$\bullet$ Comparing the conditions $\gamma<\frac{d-\alpha}{\alpha}$ in Case III and $\gamma>\frac{d}{2\alpha}$ in Case I, we note that $\frac{d-\alpha}{\alpha}>\frac{d}{2\alpha}\Longleftrightarrow\alpha<\frac{d}{2}$. Hence, when $\alpha\in(0,\frac{d}{2})$, the region $R_2$ disappears, and the boundary between the regions $R_1$ and $R_3$ is given by $\gamma=\frac{d-\alpha}{\alpha}$.

\subsection{Further comments and conjectures}\label{sec:commconj}
To conclude Section \ref{sec:main_result}, we present several comments and conjectures concerning our model and results.

\vspace{0.2cm}
\noindent\textbf{(1) About the end-to-end exponent $\xi$.}
Recall the end-to-end exponent $\xi$ from Definition \ref{def:xi}. Following the approach in \cite{BTW21}, we are able to identify the end-to-end fluctuations in the regions $R_1$, $R_2$, and $R_3$, where the contribution of the range penalty to the leading-order term of $\bZ_N$ is negligible. In the region $R_6$, we show that $\xi = 0$ (in fact, we prove the stronger result that $\bP_N(|\cR_N|=2)\overset{N\to\infty}{\longrightarrow}1$ in $\bbP$-probability).

However, we are unable to rigorously verify that the candidate exponents $\xi = \frac{1+\zeta}{d+2}\in(0,\frac1d)$ in the region $R_5$ and $\xi = \frac{\zeta}{2}\in(\frac1d,\frac12)$ in the region $R_4$ satisfy Definition \ref{def:xi}. To establish this, one needs to show that the dominant contribution to $\bZ_N$ comes from paths with $M_N \approx N^\xi$ (see Section~\ref{ch:discuss} for a more detailed discussion). In this paper, we bypass this difficulty by directly comparing $\bZ_N$ with its homogeneous counterpart $\bZ_{N,\beta_N=0}^{\omega,h_N}$.

\vspace{0.2cm}
\noindent\textbf{(2) About the regions $R_4$ and $R_5$.} The regions $R_4$ and $R_5$ correspond to sub-diffusive regimes, which are comparable to the homogeneous counterpart of our model, namely, the random walk penalized by its range. This homogeneous model was first introduced in \cite{DV1979} and has been extensively studied in recent years (see \cite{Bolt94,BC21,DFSX20,DFSX20+,DFSX21}). In that setting, $\beta_N\equiv 0$ and $h_N \equiv \hat{h}$ (corresponding to the case $\gamma=+\infty$ and $\zeta = 0$ in our model), i.e., there is no random environment and hence the partition function is deterministic.

The properties of the homogeneous model are now well understood. First,
\begin{equation}\label{eq:homo-partition}
\bZ_{N,\beta_N=0}^{\omega,h_N=\hat{h}}\sim\exp(-c_d(\hat{h})(1+o(1))N^{\frac{d}{d+2}}),\quad\text{as}~N\to\infty, 
\end{equation}
where $c_d(\hat{h})$ is given in \eqref{eq:R5scaling}. Second, under $\bP_N$, the range $\cR_N$ is asymptotically a ball of radius $r_d N^{\frac{1}{d+2}}$ (corresponding to the special case $\zeta=0$ in the region $R_5$) with no holes, where $r_d$ is explicit, and the distribution of the center of the ball has also been characterized. Moreover, a local limit theorem and large deviations for $S_N$ under $\bP_N$ have also been established.

We obtain the same asymptotics \eqref{eq:homo-partition} for our model in the region $R_5$ (see Theorem~\ref{thm:R5a}). It is therefore natural to investigate this region in greater depth and compare it with the homogeneous model. In particular, one may ask the following questions.
\begin{itemize}
\item[(2a)] Is the range $\cR_N$ also asymptotically a ball of radius $r_d N^{\frac{1+\zeta}{d+2}}$ with no holes? We conjecture that this is the case. However, it is not straightforward to verify this by directly adapting the methods in \cite{BC21,DFSX20}, due to the presence of the random environment and the need for a rigorous characterization of the radius and the absence of holes, analogous to Definition \ref{def:xi}.
\item[(2b)] If the answer to (2a) is affirmative, then how can one characterize the center $\mathcal{O}_N$ of the ball? It is known that, in the homogeneous model, the rescaled $\mathcal{O}_N$ admits a random limit (see, e.g., \cite{B24+,S90,DFSX20}). In our model, however, whether the scaling limit of $\mathcal{O}_N$ is random should depend on the strength of the disorder. If the disorder does not influence the second-order term of $\log\bZ_N$, then the behavior of $\mathcal{O}_N$ should be comparable to that in the homogeneous model. Otherwise, the scaling limit of $\mathcal{O}_N$ is expected to be non-random for a given $\omega$, in the sense that $\mathcal{O}_N$ is determined by the random environment so as to maximize the energy gain within the ball.
\item[(2c)] A question related to (2b) is to study the higher-order fluctuations of $\log\bZ_N$, in particular, the fluctuations of $\frac{1}{N^{2\xi-1}}\log\bZ_N+c_d(\hat{h})$, where $-c_d(\hat{h})$ is given in \eqref{eq:R5scaling}.
\end{itemize}
These questions have been addressed in dimension $d=1$ in \cite{B24}. However, they become significantly more challenging in dimensions $d\geq2$, due to the geometric complexity of the range of $S$.

In the region $R_4$, we obtain different asymptotics in Theorem \ref{thm:R4}, identifying a new phase transition that cannot be observed when $h_N$ is fixed. Since it is expected that, in this region, the random walk has enough space to stretch so that $|\cR_N|\approx N$, it becomes more difficult to formulate a tractable problem concerning the geometry of $\cR_N$, as in the region $R_5$. In particular, it seems impossible to define a natural notion of the center of $\cR_N$, and thus questions (2a)–(2b) are no longer well-defined.

We may conjecture that, if the random environment influences the second-order term of $\log\bZ_N$, then it encourages the random walk to visit regions with large energy gain (but still within the scale $N^\xi\ll\sqrt{N}$) and induces a perturbation in the distribution of $S$, making it less symmetric under $\bP_N$. Nevertheless, it remains open how to characterize this asymmetry in a tractable way.

Another direction is to study the fluctuation of $\frac{1}{N^{1-2\xi}}\log\bZ_N+\hat{h}\gamma_d$, where $\gamma_d$ is given in \eqref{eq:range_limit}, in analogy with question (2c) above.

\vspace{0.2cm}
\noindent\textbf{(3) About the region $R_1$.} In the region $R_1$, we identify new phase transitions in \eqref{eq:new_phase}–\eqref{eq:old_phase} and \eqref{eq:new}–\eqref{eq:old2}. Since this is the diffusive regime, it is natural to ask whether there are differences in the finer behavior of the random walk $S$ across these phases. Note that for the one-dimensional case, it was proved in \cite{BTW21} that throughout the entire diffusive regime,
\begin{equation}\label{eq:TV}
\lim\limits_{N\to\infty}d_{\mathbf{TV}}(\bP_N,\bP)=0,\quad\text{in}~\bbP\text{-probability},
\end{equation}	
where $d_{\mathbf{TV}}(\cdot,\cdot)$ denotes the total variation distance. We expect that the same result also holds in higher dimensions. Consequently, any perturbation from the random environment or the range penalty to the distribution of $S$ should not be observable at the diffusive scale $\sqrt{N}$.

In the phase corresponding to \eqref{eq:old_phase} and \eqref{eq:old2}, the random environment dominates the leading-order term of $\log\bZ_N$. However, in view of \eqref{eq:TV}, we expect that it can only induce a mild perturbation in the distribution of $S$. As in the region $R_4$, the range $\cR_N$ is not a ball without holes, and its geometry is not transparent. We therefore conjecture that the effect of the random environment is limited to inducing a mild asymmetry in $S$.

In contrast, in the phase corresponding to \eqref{eq:new_phase} and \eqref{eq:new}, the range penalty dominates the leading-order term of $\log\bZ_N$. Since the range penalty is homogeneous, we expect that the symmetry of $S$ under $\bP_N$ is preserved. However, it may still have a mild confining effect on the range $\cR_N$. Hence, one may investigate the appropriate scaling limit or fluctuations of $\bE[|\cR_N|]-\bE_N[|\cR_N|]$, which should be stochastically positive, in order to quantify this effect.

\subsection{Notations}\label{S:notations}
\begin{itemize}
\item We will sometimes consider $\bZ_{N,\beta_N=0}^{\omega,h_N}$ and $\bZ_{N,\beta_N}^{\omega,h_N=0}$, where the former is the partition function of the homogeneous model, and the latter is the partition function without range penalty studied in \cite{BTW21}. We also use the notation $\bZ_N(A):=\bE[e^{\sum_{x\in\cR_N}(\beta_N\omega_x-h_N)}\ind_A]$.
\item We define $M_N=M_N(S):=\max_{1\leq n\leq N}\|S_n\|$, which represents the maximal distance from the origin reached by the random walk up to time $N$, where $\|\cdot\|$ denotes the Euclidean norm.
\item Let $\bM^{(r)}_k$ denote the $k$-th largest value among the random variables $(\omega_x)_{x\in \bbZ^d}$ with $x\in\Lambda_r$, where $\Lambda_r:=\{x\in\bbZ^d: \|x\|\leq r\}$ is the Euclidean ball with radius $r$.
\item We may also relabel the random environment in the ball $\Lambda_r$ using order statistics, which we denote by $(\bM_i^{(r)},\bY_i^{(r)})_{i=1}^{|\Lambda_r|}$. Here $(\bY_i^{(r)})_{i=1}^{|\Lambda_r|}$ is a random permutation of $\Lambda_r$, where $\bY_i^{(r)}\in\Lambda_r$ denotes the location of the $i$-th largest random variables $\omega_x$ with $x\in\Lambda_r$.
\item When restricted to the ball $\Lambda_r$, the energy gain associated with a set $\Delta\subset\bbZ^d$ is defined by
\begin{equation}\label{eq:bigOmega}
\Omega_r(\Delta):=\sum_{i=1}^{|\Lambda_r|}\bM_i^{(r)}\mathbbm{1}_{\{\bY_i^{(r)}\in \Delta\}}.
\end{equation}
\end{itemize}

\section{Proof for Theorem \ref{thm:R4}}\label{sec:R4}
In the region $R_4$, recall that we have $d\geq3$, $\frac{2}{d}<\zeta<1$ and $\xi=\frac{\zeta}{2}$. Furthermore, if $\alpha\in(\frac{d}{2},d)$, then $\gamma ~> ~\frac{2\alpha-d}{2\alpha}\zeta~+~\frac{d-\alpha}{\alpha}$ and if $\alpha\in(0,\frac{d}{2})$, then $\gamma>\frac{d-\alpha}{\alpha}$. In particular, the critical case $\alpha=\frac{d}{2}$ can be included, where $\frac{d-\alpha}{\alpha}=\frac{d}{2\alpha}$.

We first establish an upper bound for $\frac{1}{N^{1-2\xi}}\log\bZ_N$. By H\"older's inequality, for any $\epsilon>0$,
\begin{equation}\label{eq:Holder}
\begin{split}
\bZ_N\leq
\Big(\bZ_{N, \beta_N=0}^{\omega, (1+\epsilon)h_N}\Big)^{\frac{1}{1+\epsilon}}
\Big(\bZ_{N,\frac{1+\epsilon}{\epsilon}\beta_N}^{\omega, h_N=0}\Big)^{\frac{\epsilon}{1+\epsilon}}.
\end{split}
\end{equation}
Taking logarithm and dividing it by $N^{1-2\xi}$, we have that
\begin{equation}\label{log_holder_ub}
\frac{1}{N^{1-2\xi}}\log\bZ_N\leq\frac{1}{(1+\epsilon)N^{1-2\xi}}\log\bZ_{N,\beta_N=0}^{\omega, (1+\epsilon)h_N}+\frac{\epsilon}{(1+\epsilon)N^{1-2\xi}}\log\bZ_{N, \frac{1+\epsilon}{\epsilon}\beta_N}^{\omega, h_N=0}.
\end{equation}

Similarly, we have the following lower bound. By H\"older's inequality, for any $\epsilon\in (0,1)$,
\begin{align*}
\bZ_{N, \beta_N=0}^{\omega, (1-\epsilon)h_N}&\le 
\Big(\bZ_N\Big)^{1-\epsilon}\Big(\bZ_{N, -\frac{1-\epsilon}{\epsilon}\beta_N}^{\omega, h_N=0}\Big)^\epsilon,
\end{align*}
that is,
\begin{equation}\label{eq:invers_Holder}
\bZ_N\ge \Big(\bZ_{N,\beta_N=0 }^{\omega, (1-\epsilon)h_N}\Big)^{\frac{1}{1-\epsilon}}\Big(\bZ_{N, -\frac{1-\epsilon}{\epsilon}\beta_N}^{\omega, h_N=0}\Big)^{-\frac{\epsilon}{1-\epsilon}}.
\end{equation}	
Taking logarithm and dividing it by $N^{1-2\xi}$, we have that
\begin{equation}\label{log_holder_lb}
\frac{1}{N^{1-2\xi}}\log\bZ_N\geq\frac{1}{(1-\epsilon)N^{1-2\xi}}\log\bZ_{N,\beta_N=0}^{\omega, (1-\epsilon)h_N}-\frac{\epsilon}{(1-\epsilon)N^{1-2\xi}}\log\bZ_{N, -\frac{1-\epsilon}{\epsilon}\beta_N}^{\omega, h_N=0}.
\end{equation}

We now show that the second terms on the right-hand sides of \eqref{log_holder_ub} and \eqref{log_holder_lb} are negligible (Step 1), and that the correct asymptotics are given by the first terms (Step 2).

\vspace{0.2cm}	
\noindent\textbf{Step 1.} We show that for any $\delta\in\bbR$, $\log\bZ_{N, \delta\beta_N}^{\omega, h_N=0}=o(N^{1-2\xi})$ in $\bbP$-probability. Note that we are in the setting of \cite{BTW21}, since $h_N=0$. Hence, many computations in \cite[Section~7]{BTW21} can be directly applied here. We consider three cases separately.
	
\vspace{0.2cm}
\noindent\textbf{Case 1.} For $\frac{d}{2}<\ga<d$ and $\frac{d-\ga}{\ga}<\gamma<\frac{d}{2\ga}$, 
the partition function falls in the region B in \cite{BTW21}. Hence, by \cite[Theorem~2.7]{BTW21}, 
\begin{equation*}
\frac{1}{N^{2\xi'-1}}\log\bZ_{N, \delta\beta_N}^{\omega, h_N=0}~\text{converges in distribution to a non-trivial limit with}~\xi'=\frac{\ga(1-\gamma)}{2\ga-d}.
\end{equation*}
Note that in the region $R_4$, we have that $\xi=\frac{\zeta}{2}$. Therefore, since $\gamma>\frac{2\alpha-d}{2\alpha}\zeta+\frac{d-\alpha}{\alpha}$,
\begin{equation*}
\begin{split}
1-2\xi-(2\xi'-1)=&1-\zeta-\Big(\frac{2\ga(1-\gamma)}{2\ga-d}-1\Big)=\frac{2\ga\gamma-(2\ga-d)\zeta+2(\ga-d)}{2\ga-d}\\
>&\frac{1}{2\ga-d}\Big(2\ga\Big(\frac{2\ga-d}{2\ga}\zeta+\frac{d-\ga}{\ga}\Big)-(2\ga-d)\zeta+2(\ga-d)\Big)=0.
\end{split} 
\end{equation*}
Hence, $\frac{1}{N^{1-2\xi}}\log\bZ_{N, \delta\beta_N}^{\omega, h_N=0}\longrightarrow0$ in $\bbP$-probability as $N\to\infty$.

\vspace{0.2cm}
\noindent\textbf{Case 2.} For $\frac{d}{2}\leq\ga<d$ and $\gamma\geq\frac{d}{2\ga}$, 
the partition function falls in the region C in \cite{BTW21} (the boundary between the regions B and C in \cite{BTW21} is included). We adapt the approaches in \cite[Section 7]{BTW21}, where the correct scaling for $\log\bZ_{N, \delta\beta_N}^{\omega, h_N=0}$ is determined, except for $\alpha\in\{2,\frac{d}{2}\}$. Our treatment here is simpler since we only need some upper bounds.

First, we truncate the random environment $\omega$ by $k_N=\beta_N^{-1}\sqrt{\log\log N}$ and denote $\tilde{\omega}_x:=\omega_x\ind_{\{|\omega_x|\leq k_N\}}$. Our truncation is slightly different from that in \cite{BTW21} since we include the cases $\alpha\in\{2,\frac{d}{2}\}$. We show that $|\bZ_{N,\delta\beta_N}^{\omega,h_N=0}-\bZ_{N,\delta\beta_N}^{\tilde \omega,h_N=0}|\overset{N\to\infty}{\longrightarrow}0$ in $\bbP$-probability so that we can then work on $\bZ_{N,\delta\beta_N}^{\tilde \omega,h_N=0}$. 

With $A_N~=~\log\log\log N$, we follow the ideas in \cite[Lemma 7.1, Lemma 7.4]{BTW21} to show that
\begin{itemize}
\item[(1)] $\bZ_{N,\delta\beta_N}^{\omega,h_N=0}(M_N\geq A_N\sqrt{N A_N})\overset{N\to\infty}{\longrightarrow}0$ in $\bbP$-probability.
\item[(2)] $\bZ_{N,\delta\beta_N}^{\tilde{\omega},h_N=0}(M_N\geq A_N\sqrt{N A_N})\overset{N\to\infty}{\longrightarrow}0$ in $\bbP$-probability.
\item[(3)] $|\bZ_{N,\delta\beta_N}^{\omega,h_N=0}(M_N\leq A_N\sqrt{N A_N})-\bZ_{N,\delta\beta_N}^{\tilde{\omega},h_N=0}(M_N\leq A_N\sqrt{N A_N})|\overset{N\to\infty}{\longrightarrow}0$ in $\bbP$-probability.
\end{itemize}
\noindent This completes the truncation procedure.

For (1), we in fact have
\begin{equation*}
\bbP\Big(\bZ_{N,\delta\beta_N}^{\omega,h_N=0}\big(M_N\geq A_N\sqrt{N A_N}\big)>e^{-cA_N^3}\Big)\leq cA_N^{-\nu},
\end{equation*}
where we use that $\delta\beta_N N^{\frac{d}{2\alpha}} \leq C$. This estimate is identical to \cite[Proposition~6.1]{BTW21}, with $N^\xi$ replaced by $\sqrt{N A_N}$, and we therefore omit the proof (see also equation~(7.4) in \cite{BTW21}).

\smallskip

To prove (2), denote $\lambda_N=\log\bbE[\exp(\delta\beta_N\tilde{\omega}_0)]$. By Cauchy-Schwarz inequality,
\begin{equation*}
\begin{split}
\bbE\Big[\bZ_{N,\delta\beta_N}^{\tilde{\omega},h_N=0}\big(M_N\geq A_N\sqrt{N A_N}\big)\Big]&=\bE\Big[e^{\lambda_N|\cR_N|}\ind_{\{M_N\geq A_N\sqrt{N A_N}\}}\Big]\leq\bE\Big[e^{2\lambda_N|\cR_N|}\Big]^{\frac12}\bP\big(M_N\geq A_N\sqrt{N A_N}\big)^{\frac12}.
\end{split}
\end{equation*}
By \cite[Lemma C.1]{BTW21}, we have that $|\lambda_N|\leq e^{\delta\beta_N k_N}\beta_N^{\alpha\wedge 2}$ (see equation (7.9) in \cite{BTW21}). Recall that $d\geq3$ in the region $R_4$, and we have that
\begin{equation}\label{eq:lambda_R}
|\lambda_N||\cR_N|\leq e^{c\sqrt{\log\log N}}N^{1-\frac{d}{2\alpha}(\alpha\wedge2)}\overset{N\to\infty}{\longrightarrow}0
\end{equation}
by $d>\alpha$ and $d\geq3$. The estimate then follows by the classical large deviation principle of $M_N$.

Finally, we prove (3). By a union bound and $\gamma\geq\frac{d}{2\alpha}$, we have that,
\begin{equation*}
\begin{split}
&\bbP\Big(\bZ_{N,\delta_N\beta_N}^{\omega,h_N=0}\big(M_N\leq A_N\sqrt{N A_N}\big)\neq\bZ_{N,\delta_N\beta_N}^{\tilde{\omega},h_N=0}\big(M_N\leq A_N\sqrt{N A_N}\big)\Big)\\
=&\sum\limits_{\|x\|\leq A_N\sqrt{N A_N}}\bbP(|\omega|_x\geq k_N)\leq C \big(A_N\sqrt{N A_N}\big)^{d}\big(N^{\gamma}\sqrt{\log\log N}\big)^{-\alpha}\overset{N\to\infty}{\longrightarrow}0,
\end{split}
\end{equation*} where $C>0$.

Then we turn to $\bZ_{N,\delta\beta_N}^{\tilde{\omega},h_N=0}$. By a polynomial chaos expansion (see also equation (7.11) in \cite{BTW21}), we have that, with $\epsilon_x:=\exp(\delta\beta_N\tilde{\omega}_x-\lambda_N)-1$,
\begin{equation*}
\begin{split}
\bZ_{N,\delta\beta_N}^{\tilde{\omega},h_N=0}&=\bE\Big[e^{\lambda_N|\cR_N|}\Big]+\sum\limits_{k\geq1}\sum\limits_{(x_1,\cdots,x_k)\in(\bbZ^d)^k\atop x_i\neq x_j, i\neq j}\Big(\prod\limits_{i=1}^k \epsilon_{x_i}\Big)\bE\Big[e^{\lambda_N|\cR_N|}\ind_{\{(x_1,\cdots,x_k)\in\cR_N\}}\Big]=:\bE\Big[e^{\lambda_N|\cR_N|}\Big]+\bY_N.
\end{split}
\end{equation*}

For $d\geq3$, it has been shown by \eqref{eq:lambda_R} that $\lim_{N\to\infty}\bE[e^{\lambda_N|\cR_N|}]=1$, and by equation (7.14) in \cite{BTW21} and the choice of $k_N$, $\bbE[(\bY_N)^2]\leq Ce^{c\beta_N k_N}\beta_N^{\alpha\wedge2} N\overset{N\to\infty}{\longrightarrow}0$ since $d>\alpha$ and $d\geq3$. Hence, $\frac{1}{N^{1-2\xi}}\log\bZ_{N,\delta\beta_N}^{\tilde{\omega},h_N=0}\overset{N\to\infty}{\longrightarrow}0$ in $\bbP$-probability, simply due to $1-2\xi>0$.

\vspace{0.2cm}
\noindent\textbf{Case 3.} For $0<\ga<\frac{d}{2}$ and $\gamma>\frac{d-\ga}{\ga}>\frac{d}{2}$, then all the computations for \textbf{Case 2} carries over here without any changes, since the partition function still falls in the region C in \cite{BTW21}. 

\vspace{0.2cm}
\noindent\textbf{Step 2.} We show the convergence in Theorem \ref{thm:R4} by showing that for any $\vartheta>0$,
$$
\lim_{N\to +\infty}\frac1{N^{1-2\xi}}\log\bZ_{N,\beta_N=0}^{\omega, \theta h_N}=-\theta\hat{h}\gamma_d.
$$
The result then follows by taking $N\to\infty$ and then taking $\epsilon\to0$ in \eqref{log_holder_ub} and \eqref{log_holder_lb}, where $\vartheta$ is $1+\epsilon$ and $1-\epsilon$ respectively.

For the lower bound, since $f(x)=e^{-x}$ is convex, by Jensen's inequality, we have that
\begin{equation}\label{eq:R4_lower}
\bZ_{N,\beta_N=0}^{\omega, \theta h_N}=\bE\Big[e^{-\theta h_N|\cR_N|}\Big]\ge e^{-\theta h_N\bE[|\cR_N|]}.
\end{equation}
The lower bound then follows by $\xi=\frac{\zeta}{2}$ and \eqref{eq:range_limit}.
	
For the upper bound, we write
\begin{align*}
\bZ_{N,\beta_N=0}^{\omega, \theta h_N}	&=e^{-\theta h_N\bE[|\cR_N|]} \bE\Big[e^{-\theta h_N(|\cR_N|-\bE[|\cR_N|])}\Big]
\end{align*}
and then we only need to show that $\bE[e^{-\theta h_N(|\cR_N|-\bE[|\cR_N|])}]$ is negligible. Recall that in the region $R_4$, $d\geq 3$. By \cite[Theorem 1]{BBH01}, we have the following large deviation principle for $|\cR_N|$: for any $\delta>0$,
\begin{equation}\label{LDP_range}
\lim_{N\to\infty}\frac{1}{N^{1-\frac{2}{d}}}\log\bP\Big(|\cR_N|-\bE[|\cR_N|]\le -\delta N\Big)=-\kappa,
\end{equation}
where $\kappa=\kappa(d,\delta)>0$ is explicit. Set 
\begin{equation}\label{eq:goodevent}
\cG_{N, \delta}:=\Big\{|\cR_N|-\bE[|\cR_N|]\le -\delta N\Big\}
=\Big\{e^{-\theta h_N(|\cR_N|-\bE[|\cR_N|])}\ge e^{\delta\theta\hat h N^{1-\zeta}}
\Big\}.
\end{equation}
Note that $e^{-\theta h_N(|\cR_N|-\bE[|\cR_N|])}\le e^{\theta\hat h N^{1-\zeta}}$. Thus, for any $\delta>0$, there exists some $\kappa'>0$, such that
\begin{align*}
\bE\Big[e^{-\theta h_N(|\cR_N|-\bE[|\cR_N|])}\Big]
&=\bE\Big[e^{-\theta h_N(|\cR_N|-\bE[|\cR_N|])}\ind_{\cG_{N, \delta}^c} 
\Big] + \bE\Big[e^{-\theta h_N(|\cR_N|-\bE[|\cR_N|])}\ind_{\cG_{N, \delta}} 
\Big]\\
&\le e^{\delta\theta\hat h N^{1-\zeta}}+e^{\theta\hat h N^{1-\zeta}}e^{-\kappa' N^{1-\frac{2}{d}}}
\le  e^{\delta \hat h N^{1-\zeta}}+e^{-\kappa' N^{1-\frac{2}{d}}},
\end{align*}
for $N$ is large enough, since $1-\frac2d>1-\zeta$ in the region $R_4$ by $\zeta>\frac{2}{d}$. Hence,
\begin{equation*}
\limsup\limits_{\delta\to0}\limsup\limits_{N\to\infty}\frac{1}{N^{1-2\xi}}\bE\Big[e^{-\theta h_N(|\cR_N|-\bE[|\cR_N|])}\Big]\leq\limsup\limits_{\delta\to0}\limsup\limits_{N\to\infty}\frac{1}{N^{1-\zeta}}\log\Big(e^{\delta \hat h N^{1-\zeta}}+e^{-\kappa' N^{1-\frac{2}{d}}}\Big)=0,
\end{equation*}
and the upper bound follows.

\vspace{0.2cm}
Finally, we show that $\bP_N\big(\big|\frac{|\cR_N|}{N}-\gamma_d\big|<\delta\big)\overset{N\to\infty}{\longrightarrow}1$ in $\bbP$-probability. For any $\delta>0$,
\begin{equation}\label{R4_prob}
\bZ_N=\bZ_N\Big(\big||\cR_N|-\gamma_d N\big|<\delta N\Big)+\bZ_N\Big(|\cR_N|-\gamma_d N\leq-\delta N\Big)+\bZ_N\Big(|\cR_N|-\gamma_d N\geq\delta N\Big).
\end{equation}
Note that $\bZ_N(|\cR_N|-\gamma_d N\geq\delta N)\leq e^{-(\gamma_d+\delta)\hat{h}N^{1-\zeta}}\bZ_{N,\beta_N}^{\omega,h_N=0}$ and by Cauchy-Schwarz inequality,
\begin{equation*}
\Big(\bZ_N\Big(|\cR_N|-\gamma_d N\leq-\delta N\Big)\Big)^2\leq \bZ_{N,\beta_N=0}^{\omega, 2h_N}\Big(|\cR_N|-\gamma_d N\leq-\delta N\Big)\bZ_{N,2\beta_N}^{\omega,h_N=0}\leq e^{-\kappa' N^{1-\frac{2}{d}}}\bZ_{N,2\beta_N}^{\omega,h_N=0}.
\end{equation*}
We have already shown that for any $\epsilon\in(0,\delta)$, with $\bbP$-probability arbitrarily close to $1$, $\log\bZ_{N, \beta_N}^{\omega,h_N=0}=o(N^{1-\zeta})$ (see Step 1) and $\bZ_N>\exp(-(\gamma_d+\epsilon)\hat{h}N^{1-\zeta})$ (by \eqref{eq:R4_lower}). Hence, dividing the both sides of \eqref{R4_prob} by $\bZ_N$, we show that by $1-\frac{2}{d}>1-\zeta$, 
\begin{equation*}
\bP_N\Big(\Big|\frac{|\cR_N|}{N}-\gamma_d\Big|<\delta\Big)\geq 1-e^{-(\delta-\epsilon)\hat{h}N^{1-\zeta}}-e^{-\frac{1}{2}\kappa' N^{1-\frac{2}{d}}}\overset{N\to\infty}{\longrightarrow}1.
\end{equation*}

\section{Proof for Theorem \ref{thm:R6}}\label{sec:R6}
In the region $R_6$, recall that $\zeta<-1$ and $\gamma>\zeta+\frac{d-\alpha}{\alpha}$. We write 
\begin{equation}\label{eq:R6split}
\bZ_N=
\bZ_N(|\cR_N|=2)+\bZ_N(|\cR_N|\ge 3).
\end{equation}
Note that the critical case $\alpha=\frac{d}{2}$ can also be included.

For the first term,
\begin{equation*}
\bZ_N(|\cR_N|=2)= \exp(-2 h_N)
\sum_{k=1}^d\sum\limits_{j=\pm1}e^{\beta_N(\omega_0+\omega_{je_k})}\bP\big(\cR_N=\{0,je_k\}\big),
\end{equation*}
where $e_k$ is a $d$-dimensional vector with the $k$-th element being 1 and other elements being 0.

Note that $\bP(\cR_N=\{0,je_k\})=(1/2d)^{N}$ for all $k,j$ above. By $\zeta<-1$, we get that 
\begin{equation}\label{eq:R6_main}
\lim_{N\to\infty} N^{\zeta}\log \bZ_N(|\cR_N|=2)=-2\hat h,\quad\text{in}~\bbP\text{-probability},
\end{equation}
since by $\gamma>\zeta+\frac{d-\alpha}{\alpha}$, there exists some $\epsilon>0$, such that by a union bound and \eqref{def:omega},
\begin{equation*}
\bbP\Bigg(\beta_N\Big(|\omega_0|+\sum_{k=1}^d(|\omega_{e_k}|+|\omega_{-e_k}|)\Big)>N^{-\zeta-\epsilon}\Bigg)<C_d N^{-\alpha(\gamma-\zeta-\epsilon)}\overset{N\to\infty}{\longrightarrow}0.
\end{equation*}

Then we show that the second term in \eqref{eq:R6split} is negligible. By \eqref{eq:R6_main}, it suffices to show that
\begin{equation}\label{eq:>3}
\lim_{N\to +\infty} \bbP\Big(\bZ_N(|\cR_N|\ge 3) > e^{-\frac{5}{2}\hat h N^{-\zeta}}\Big)=0.
\end{equation}
Recall $\bM_{1}^{(A_N)}$ and $M_N$ from Section \ref{S:notations}, where  $\bM_{1}^{(A_N)}$ denotes the largest value among the random variables $(\omega_x)_{x \in \bbZ^d}$ with $|x|\leq A_N$ and $M_N:=\max_{1\leq n\leq N}\|S_n\|$. We choose $A_N=N^{1-\frac{\alpha}{d}}\overset{N\to\infty}{\longrightarrow}\infty$ and split the partition function as follows
\begin{align}\label{eq:R6split2}
\bZ_N(|\cR_N|\ge 3) \le \bZ_N(|\cR_N|\ge 3, M_N\le A_N)+ \bZ_N(M_N>A_N).
\end{align}

For the first term above, 
\begin{align*}
\bZ_N(|\cR_N|\ge 3, M_N\le A_N) = \sum_{k=3}^N \bZ_N(|\cR_N|=k, M_N\le A_N)\le 
\sum_{k=3}^N e^{-k\hat h N^{-\zeta}+k\hat \beta  N^{-\gamma}\bM_{1}^{(A_N)}}.
\end{align*}
We have that by a large deviation estimate of heavy-tailed random variables (see \cite[Lemma 5.1]{BT19}),
\begin{equation}\label{eq:large_dev}
\bbP\Big( \bM_{1}^{(A_N)} > \frac{\hat h}{6\hat \beta}  N^{\gamma-\zeta}\Big) 
\le \Big( c_{\hat{h},\hat{\beta}} N^{\gamma-\zeta} A_N^{-\frac{d}{\alpha}}\Big)^{-\alpha}=cN^{-\ga(\gamma-\zeta-\frac{d-\ga}{\ga})}\overset{N\to\infty}{\longrightarrow}0.
\end{equation}
Therefore, with $\bbP$-probability close to $1$,
\begin{equation}\label{R6:<A_N}
\bZ_N(|\cR_N|\ge 3, M_N\le A_N)\leq \sum_{k=3}^N e^{-k \frac56 \hat h N^{-\zeta} } \le e^{-\frac52 \hat h N^{-\zeta}}.
\end{equation}

We then turn to $\bZ_N(M_N>A_N)$ in \eqref{eq:R6split2}. We have that 
\begin{equation}\label{eq:Z_Nalpha>2M>A}
\begin{split}
\bbP\Big(\bZ_N(M_N> A_N) > e^{-\frac{5}{2}h_N}\Big)
&=
\bbP\Big(\bE\big[e^{\beta_N \sum_{x\in\cR_N}\omega_x-h_N|\cR_N|+\frac52 h_N}\mathbbm{1}_{\{M_N> A_N\}}\big] > 1\Big) \\
&\le \bbP\Big(\bZ_{N,\beta_N}^{\omega,\frac12 h_N}(M_N>A_N) > 1\Big).
\end{split}
\end{equation}
We show that the last probability above converges to $0$ as $N\to\infty$. By the crucial Lemma \ref{lem:g>zda} below, we will show in a while that
\begin{equation}\label{eq:R6}
\lim\limits_{N\to\infty}\bbP\Big(\bZ_{N,\beta_N}^{\omega,\frac12 h_N}(M_N>A_N)\ge \bZ_{N,\beta_N=0}^{\omega,\frac14 h_N}(M_N> A_N)\Big)=0.
\end{equation}
Then we can conclude our claim for \eqref{eq:Z_Nalpha>2M>A} by a trivial bound, $\bZ_{N,\beta_N=0}^{\omega,\frac14 h_N}(M_N> A_N)\le e^{-\frac{1}{4}\hat{h}N^{-\zeta+1-\frac\alpha{d}}}\ll1$.

Now \eqref{eq:>3} follows by \eqref{eq:R6split2}, \eqref{R6:<A_N}, \eqref{eq:Z_Nalpha>2M>A} and \eqref{eq:R6}. Furthermore, by \eqref{eq:R6split}
\begin{equation*}
1=\bP_N(|\cR_N|=2)+\frac{\bZ_N(|\cR_N|\geq3)}{\bZ_N},
\end{equation*}
where the second term is bounded by $\exp(-\frac14\hat{h}N^{-\zeta})$ with $\bbP$-probability close to $1$ as $N\to\infty$. Hence, $\bP_N(|\cR_N|=2)\to1$ as $N\to\infty$ in $\bbP$-probability.

It remains to prove \eqref{eq:R6}. To this end, we require the crucial Lemma~\ref{lem:g>zda}, which follows from the analysis of last passage percolation with a prescribed range (see Section \ref{sec:appA}) and allows us to control the energy gain within a given box and for a given range. We first state and prove the lemma, and then use it to establish \eqref{eq:R6}.

\begin{lemma}\label{lem:g>zda}
Recall $\Omega_r(\Delta)$ from Section \ref{S:notations}. If $\gamma>\zeta+\frac{d-\alpha}{\alpha}$, $A_N\le N$, 
then for any $\epsilon\in (0,1)$ and $p>1$, there exists a constant $C=C_{d,p,\epsilon}>0$, such that
\begin{multline}\label{sumprob}
\sum_{k=0}^{\log_p(\frac{N}{A_N})}\sum_{j=0}^{B_{N,k}}\bbP\Bigg(\max\limits_{S\in \cA_{j,k} }\Omega_{p^{k+1}A_N}(\cR_N(S))>\epsilon^2p^{k+j}\frac{h_N}{\beta_N}A_N\Bigg)
\le C(\log N) (N^{-\zeta+\gamma-\frac{d-\alpha}\alpha})^{-\frac{\alpha d}{\alpha+d}},
\end{multline}
where
\begin{equation}\label{eq:B_Nk}
B_{N,k}:=\log_p\Big(\frac{|\Lambda_{p^k A_N}|}{A_N}\wedge \frac{N}{p^k A_N}\Big)
\end{equation}
and $\cA_{j,k}=\cA_{j,k}(p,N) := \cM_p(p^kA_N, p^j)$ (see \eqref{eq:Mrs}), that is,
\begin{equation}\label{def:Ajk}
\cA_{j,k} =\Big\{S\colon \,  M_N(S)\in [p^kA_N, p^{k+1}A_N],\, |\cR_N(S)|\in [p^{k}p^{j}A_N ,p^{k+1}p^{j+1}A_N]\Big\}.
\end{equation}
\end{lemma}

\begin{proof}
By taking $r=p^k A_N$ and $s=p^j$ in Proposition \ref{lem:Omega}, the probability in \eqref{sumprob} is smaller than 
\begin{equation}
\sum_{k=0}^{\log_p(\frac{N}{A_N})}\sum_{j=0}^{B_{N,k}}
C\Big(\epsilon^2 \frac{h_N}{\beta_N}A_N^{1-\frac{d}{\alpha}} p^{(1-\frac{d}{\alpha})k}\Big)^{-\frac{\alpha d}{\alpha+d}}\le 
C\log N \Big(\epsilon^2 \frac{h_N}{\beta_N}N^{1-\frac{d}{\alpha}}\Big)^{-\frac{\alpha d}{\alpha+d}},
\end{equation}
where we have used that  $B_{N,k}\leq C_d\log N$ and 
\begin{equation*}
\sum\limits_{k=0}^{\log_p(\frac{N}{A_N})}p^{\frac{d(d-\alpha)}{\alpha+d}k}\leq C_p\Big(\frac{N}{A_N}\Big)^{\frac{d(d-\alpha)}{\alpha+d}}.
\end{equation*}
\end{proof}
\begin{proof}[Proof for \eqref{eq:R6}] We first decompose $\bZ_{N,\beta_N}^{\omega,\frac12 h_N}(M_N>A_N)$ and $\bZ_{N,\beta_N=0}^{\omega,\frac14 h_N}(M_N>A_N)$ by
\begin{align*}
\bZ_{N,\beta_N}^{\omega,\frac12 h_N}(M_N>A_N)&=\sum\limits_{k=0}^{\log_p(\frac{N}{A_N})}\sum\limits_{j=0}^{B_{N,k}}\bZ_{N,\beta_N}^{\omega,\frac12 h_N}(\cA_{j,k}),\\
\bZ_{N,\beta_N=0}^{\omega,\frac14 h_N}(M_N>A_N)&=\sum\limits_{k=0}^{\log_p(\frac{N}{A_N})}\sum\limits_{j=0}^{B_{N,k}}\bZ_{N,\beta_N=0}^{\omega,\frac14 h_N}(\cA_{j,k}),
\end{align*}
for some $p>1$, where $B_{N,k}$ is defined in \eqref{eq:B_Nk}.

Then by a union bound and \eqref{sumprob}, we obtain that the probability in \eqref{eq:R6} is bounded from above by
\begin{equation}\label{eq:last_step}
\begin{split}
&\sum\limits_{k=0}^{\log_p(\frac{N}{A_N})}\sum\limits_{j=0}^{B_{N,k}}\bbP\Bigg(\bZ_{N,\beta_N}^{\omega,\frac12 h_N}(\cA_{j,k})\geq \bZ_{N,\beta_N=0}^{\omega,\frac14 h_N}(\cA_{j,k})\Bigg)\\
\leq&\sum\limits_{k=0}^{\log_p(\frac{N}{A_N})}\sum\limits_{j=0}^{B_{N,k}}\bbP\Bigg(-\frac12 p^{k+j}h_NA_N+\beta_N\max\limits_{S\in\cA_{j,k} }\Omega_{p^{k+1}A_N}(\cR_N(S))>-\frac14 p^{k+j+2}h_NA_N\Bigg)\\
\leq&\sum\limits_{k=0}^{\log_p(\frac{N}{A_N})}\sum\limits_{j=0}^{B_{N,k}}\bbP\Bigg(\max\limits_{S\in \cA_{j,k} }\Omega_{p^{k+1}A_N}(\cR_N(S))>\frac14(2-p^2) p^{k+j}\frac{h_N}{\beta_N}A_N\Bigg)\\
\leq&C(\log N) (N^{-\zeta+\gamma-\frac{d-\alpha}\alpha})^{-\frac{\alpha d}{\alpha+d}}\to0,\quad\text{as}~N\to\infty,
\end{split}
\end{equation}
where $\epsilon=\frac12\sqrt{2-p^2}$ when $1<p<\sqrt{2}$ and we have used that $\gamma>\zeta+\frac{d-\alpha}{\alpha}$ in the region $R_6$.

\end{proof}

\section{Proof for Theorem \ref{thm:R5a}}\label{sec:R5}
Let us begin this section by observing that, using the Donsker-Varadhan argument \cite{DV1979} (although the authors only treated the case $h_N=\hat{h}$), for $\zeta<\frac2d$ (i.e. $\xi\in(0,\frac1d)$), we obtain that
\begin{equation}\label{eq:varadhan}
\bZ_{N,\beta_N=0}^{\omega,h_N}=\exp\Big(-c_d(\hat h) N^\frac{d-2\zeta}{2+d}(1+o(1)\Big).
\end{equation}

Recall that $\frac{d-2\zeta}{2+d}=1-2\xi$ in the region $R_5$. This yields the following convergence for the homogeneous partition function
$$
\lim_{N\to\infty}\frac{1}{N^{1-2\xi}}\log\bZ_{N,\beta_N=0}^{\omega,h_N}=c_d(\hat{h}).
$$
Our strategy is therefore to show that $\bZ_N$ and $\bZ_{N,\beta_N=0}^{\omega,h_N}$ are comparable. Note that the critical case $\alpha=\frac{d}{2}$ can be included in cases (1) and (3).

\begin{proof}[Proof for the case (1) of Theorem \ref{thm:R5a}]
In this case, we have that $\gamma>\zeta+\frac{d-\alpha}{\alpha}$. The proof is similar to that for Theorem \ref{thm:R6}. Let $A_N=N^{1-\frac{\alpha}{d}}$. We write
\begin{equation}\label{splitZR5}
\bZ_N=\bZ_N(M_N\le A_N)+\bZ_N(M_N> A_N).
\end{equation}

Recall $\bM^{(A_N)}_1$ and $M_N$ from Section \ref{S:notations}. We also use $\mathrm{m}^{(A_N)}_1$ to denote the smallest value among the random variables $\omega_x$ with $\|x\|\leq A_N$. Note that $\bbP(\bM^{(A_N)}_1\leq0)=\bbP(\omega\leq0)^{|\Lambda_{A_N}|}\overset{N\to\infty}{\longrightarrow}0$, and similarly $\bP(\mathrm{m}^{(A_N)}_1\geq0)\overset{N\to\infty}{\longrightarrow}0$. Hence, for the first term above, we have that with $\bbP$-probability close to $1$,
\begin{align*}
\bZ_N(M_N\le A_N)=\sum_{k\ge 1}\bZ_N(M_N\le A_N, |\cR_N|=k)\left\{\begin{aligned}&\le \sum_{k\ge 1}e^{-kh_N+k\beta_N  \bM_{1}^{(A_N)}}\bP(M_N\le A_N, |\cR_N|=k),\\
&\ge \sum_{k\ge 1}e^{-kh_N+k\beta_N \mathrm{m}^{(A_N)}_1}\bP(M_N\le A_N, |\cR_N|=k).
\end{aligned}\right.
\end{align*}
Moreover, since the right-tail asymptotics of $\bM_{1}^{(A_N)}$ and the left-tail asymptotics of $\mathrm{m}_1^{(A_N)}$ differ by at most a constant factor (see \eqref{def:omega}), it follows that, for any $\epsilon > 0$, we have estimates analogous to \eqref{eq:large_dev}:
\begin{align*}
\max\bigg\{\bbP\Big(\mathrm{m}^{(A_N)}_1<-\epsilon\frac{h_N}{\beta_N}\Big),\bbP\Big( \bM_{1}^{(A_N)}>\epsilon \frac{h_N}{\beta_N}\Big)\bigg\}
\le c(N^{\gamma-\zeta}A_N^{-\frac{d}\alpha})^{-\alpha}\to 0, \quad\text{as}~N\to +\infty.
\end{align*}
Therefore, with $\bbP$-probability close to $1$,
\begin{equation}\label{bound1}
\bZ_{N,\beta_N=0}^{\omega, (1+\epsilon)h_N}(M_N\le A_N)\le \bZ_N(M_N\le A_N) \le \bZ_{N,\beta_N=0}^{\omega, (1-\epsilon)h_N}(M_N\le A_N).
\end{equation}

For the second term in \eqref{splitZR5}, we are going to show that with $\bbP$-probability close to $1$,
\begin{equation}\label{bound2} 
\bZ_{N,\beta_N=0}^{\omega, (1+\epsilon)h_N}(M_N>A_N)\le \bZ_N(M_N> A_N) \le \bZ_{N,\beta_N=0}^{\omega, (1-\epsilon)h_N}(M_N> A_N).
\end{equation}
The strategy is similar to that used in the proof of \eqref{eq:R6}, and we will again apply Lemma~\ref{lem:g>zda}. Take $p=1+\frac{\epsilon}{10}>1$. By a union bound (see \eqref{eq:B_Nk}, \eqref{def:Ajk} and \eqref{eq:last_step}),
\begin{equation}\label{decomp}
\begin{split}
&\bbP\left(\bZ_N(M_N>A_N)>\bZ_{N,\beta_N=0}^{\omega,(1-\epsilon)h_N}(M_N>A_N)\right)\\
\leq&\sum_{k=0}^{\log_p(\frac{N}{A_N})}\sum_{j=0}^{B_{N,k}}\bbP\Big(\bZ_N(\cA_{j,k})>\bZ_{N,\beta_N=0}^{\omega,(1-\epsilon)h_N}(\cA_{j,k})\Big)\\
\leq&\sum_{k=0}^{\log_p(\frac{N}{A_N})}\sum_{j=0}^{B_{N,k}}\bbP\Bigg(\max\limits_{S\in \cA_{j,k} }\Omega_{p^{k+1}A_N}(\cR_N(S))>\epsilon^2p^{k+j}\frac{h_N}{\beta_N}A_N\Bigg),
\end{split}
\end{equation}
where we have used that $1-p^2(1-\epsilon)>\epsilon^2$ for $\epsilon$ sufficiently small. Then it follows by Lemma \ref{lem:g>zda} that \eqref{decomp} tends to $0$ as $N\to\infty$.

Similarly, noting that
\begin{equation*}
\begin{split}
\bbP\Big(\bZ_N(\cA_{j,k})<\bZ_{N,\beta_N=0}^{\omega,(1+\epsilon)h_N}(\cA_{j,k})\Big)&\leq\bbP\Bigg(\min\limits_{S\in \cA_{j,k} }\Omega_{p^{k+1}A_N}(\cR_N(S))<-\epsilon^2p^{k+j}\frac{h_N}{\beta_N}A_N\Bigg)\\
&\leq\bbP\Bigg(\max\limits_{S\in \cA_{j,k} }-\Omega_{p^{k+1}A_N}(\cR_N(S))>\epsilon^2p^{k+j}\frac{h_N}{\beta_N}A_N\Bigg),
\end{split}
\end{equation*}
where we have used that $p^2-(1+\epsilon)<-\epsilon^2$ for $\epsilon$ sufficiently small. Note that under our assumption \eqref{def:omega}, the right-tail asymptotics of $-\omega$ and $\omega$ differ by at most a constant factor. Hence, the same argument as in \eqref{decomp} also yields
\begin{equation*}
\lim\limits_{N\to\infty}\bP\Big(\bZ_N(M_N>A_N)>\bZ_{N,\beta_N}^{\omega,(1+\epsilon)h_N}(M_N>A_N)\Big)=1.
\end{equation*} 

Using \eqref{bound1}, \eqref{bound2} and \eqref{eq:varadhan}, by sending $N\to\infty$ first and then sending $\epsilon\to0$, we obtain \eqref{eq:R5scaling}.
\end{proof}

\begin{proof}[Proof for the cases (2) and (3) of Theorem \ref{thm:R5a}]
The proof is identical to that for Cases~2 and~3 in Theorem~\ref{thm:R4}, since the corresponding ranges of $\gamma$ coincide. We apply H\"{o}lder's inequality and then treat $\bZ_{N,\beta_N=0}^{\omega,h_N}$ and $\bZ_{N,\beta_N}^{\omega,h_N=0}$ separately. The details are then omitted.
\end{proof}

\begin{remark}\label{rmk:R5}
For $\alpha\in(\frac d2,d)$ and $(\zeta,\gamma)\in R_5$ with $\gamma\in(\frac{d-\alpha}{\alpha},\frac{d}{2\alpha})$, the approach for the case 1 in Theorem \ref{thm:R4} does not work. If we require $1-2\xi>2\xi'-1$ with $\xi=\frac{1+\zeta}{d+2}$ and $\xi'=\frac{\alpha(1-\gamma)}{2\alpha-d}$, then we should have $\gamma>\frac{2\alpha-d}{\alpha(2+d)}\zeta+\frac{d^2-\alpha d+d}{\alpha(2+d)}$, which is not always true in the region $R_5$. In particular, considering the boundary between the regions $R_5$ and $R_2$ in Figure \ref{diagram2}, we have that $\frac{2\alpha-d}{2\alpha}\zeta+\frac{d-\alpha}{\alpha}>\frac{2\alpha-d}{\alpha(2+d)}\zeta+\frac{d^2-\alpha d+d}{\alpha(2+d)}$ if and only if $\zeta>\frac{2}{d}$, which exactly happens in the region $R_4$.
\end{remark}

\section{Discussion of the remaining part in the region $R_5$ and the end-to-end exponent $\xi$}\label{ch:discuss}
We begin by discussing the remaining part of the region $R_5$, where $\alpha\in(\frac{d}{2},d)$, $\zeta\in(0,\frac{2}{d})$, and
\begin{equation*}
\frac{2\alpha-d}{2\alpha}\zeta+\frac{d-\alpha}{\alpha}<\gamma\leq\zeta+\frac{d-\alpha}{\alpha}.
\end{equation*}
In view of the current approach in the region $R_5$, we show that the energy gain is negligible compared to the range penalty uniformly over all configurations of the random walk. However, this approach breaks down when $\gamma \leq \zeta + \frac{d - \alpha}{\alpha}$, since the range penalty is no longer sufficiently strong (see also Remark~\ref{rmk:R5}).

More precisely, the range penalty attains its maximum when $|\cR_N| = N$, in which case the largest energy gain that can be absorbed by the range penalty is of order $N^{1-\zeta}$. This suggests that we require $N^{-\gamma+\frac{d}{\alpha}}\ll N^{1-\zeta}$ (with $\xi=1$), i.e., $\gamma>\zeta+\frac{d-\alpha}{\alpha}$.

Such a region also arises in the one-dimensional model. In \cite{BHTW20}, this difficulty is overcome by splitting the range of $M_N=\max_{1\leq n\leq N}\|S_n\|$ into two parts and treating each part with a different method. When $M_N$ is small, the energy gain can still be controlled by the range penalty; when $M_N$ is large, it must instead be controlled by the entropy cost.

Inspired by the choice of parameters in \cite{BHTW20,BTW21}, the threshold separating the two parts above should be of order $N^{1-\frac{\zeta}{2}}$ (see the heuristics below). Then we write
\begin{equation}\label{R5:decom}
\begin{split}
\bZ_N\Big(M_N\leq BN^\xi\Big)\leq \bZ_N&=\bZ_N\Big(M_N< AN^{1-\frac{\zeta}{2}}\Big)+ \bZ_N\Big(M_N \geq AN^{1-\frac{\zeta}{2}}\Big)\\
&\leq \bZ_N\Big(M_N< AN^{1-\frac{\zeta}{2}}\Big)+ \bZ_{N,\beta_N}^{\omega,h_N=0}\Big(M_N \geq AN^{1-\frac{\zeta}{2}}\Big),
\end{split}
\end{equation}
where $A, B>0$ are large but fixed constants. Note that $1-\frac{\zeta}{2}>\xi$ since
\begin{equation*}
\begin{split}
1-\frac{\zeta}{2}-\xi=\frac{2(1+d)-(4+d)\zeta}{2(2+d)}>\frac{2d^2-8}{2d(2+d)}\geq0
\end{split}
\end{equation*}
by $\zeta<\frac{2}{d}$ and $d\geq2$. By \cite[Theorem A]{DFSX20}, it is known that for the homogeneous model, $\bP_{N,\beta_N=0}^{\omega,h_N}(M_N\leq BN^\xi)$ can be made arbitrarily close to $1$ by choosing $B$ and $N$ large. Consequently, $\bZ_{N,\beta_N=0}^{\omega,h_N}\approx\bZ_{N,\beta_N=0}^{\omega,h_N}(M_N\leq BN^\xi)$. Hence, to treat the remaining part of the region $R_5$, it suffices to show that
\begin{itemize}\itemsep -1pt
\item[(a)] $\bZ_N(M_N\leq BN^\xi)\approx\bZ_{N,\beta_N=0}^{\omega,h_N}(M_N\leq BN^\xi)$,
\item[(b)] $\bZ_N(M_N< AN^{1-\frac{\zeta}{2}})\approx\bZ_{N,\beta_N=0}^{\omega,h}(M_N< AN^{1-\frac{\zeta}{2}})$,
\item[(c)] $\bZ_{N,\beta_N}^{\omega,h_N=0}(M_N \geq AN^{1-\frac{\zeta}{2}})$ is negligible.
\end{itemize}

We start by part(c). Let $\xi'=1-\frac{\zeta}{2}$. Since $\zeta\in(0,\frac{2}{d})$, we have $\xi'\in(\frac{1}{2},1)$, so the event $\{M_N \geq A N^{\xi'}\}$ corresponds to a super-diffusive behavior. Note that $\gamma>\frac{2\alpha-d}{2\alpha}\zeta+\frac{d-\alpha}{\alpha}$. It follows that $\frac{d\xi'}{\alpha}-\gamma<2\xi'-1$, that is, the energy gain is dominated by the entropy cost. By applying \cite[Proposition 6.1]{BTW21} to $\xi'$, we obtain that uniformly in $N$, there exist constants $c_1,c_2,\nu>0$, such that for any $A\geq1$,
\begin{equation*}
\bbP\Big(\bZ_{N,\beta_N}^{\omega,h_N=0}\Big(M_N\geq AN^{1-\frac{\zeta}{2}}\Big)\geq e^{-c_1 AN^{1-\zeta}}\Big)\leq c_2A^{-\nu}.
\end{equation*}
Hence, by taking $A$ large, with $\bbP$-probability close to $1$,
\begin{equation*}
\bZ_{N,\beta_N}^{\omega,h_N=0}\Big(M_N\geq AN^{1-\frac{\zeta}{2}}\Big)\leq e^{-c_1 AN^{1-\zeta}}.
\end{equation*}
Since $1-\zeta>1-2\xi$ by $\zeta<\frac{2}{d}$, $\bZ_{N,\beta_N}^{\omega,h_N=0}(M_N\geq AN^{1-\frac{\zeta}{2}})$ is negligible.

While part (c) is tractable, parts (a) and (b) are substantially more challenging due to the geometric complexity of the range $\cR_N$, as we illustrate below.

First, we may write
\begin{equation*}
\bZ_N\Big(M_N<AN^{1-\frac{\zeta}{2}}\Big)=\bZ_N\Big(M_N\leq BN^\xi\Big)+\bZ_N\Big(M_N\in\big(BN^\xi, AN^{1-\frac{\zeta}{2}}\big)\Big).
\end{equation*}
Here $B>2r_d$, where $r_d$ is an explicit constant such that $r_d N^\xi$ is the typical radius of $\cR_N$ under the homogeneous polymer measure (see again \cite[Theorem A]{DFSX20}). Note that the first term on the right-hand side corresponds to part (a), and we expect to show that the second term is also negligible.

A classical approach is the method of \textit{decomposition by layers} (see \cite{BT19a,BTW21}), whereby we decompose the partition function as follows:
\begin{equation}\label{eq:term_extra}
\bZ_N\Big(M_N\in\big(BN^\xi,AN^{1-\frac{\zeta}{2}}\big)\Big)=\sum\limits_{k=0}^{\ln(\frac{A}{B}N^{1-\frac{\zeta}{2}-\xi})}\bZ_N\Big(M_N\in\big[2^k B N^\xi, 2^{k+1}B N^\xi \big)\Big).
\end{equation}

We now analyze each term in the summation above at a heuristic level.
\begin{itemize}\itemsep -1pt
\item[(1)] The energy gain is of order $\exp\{(2^{k+1}BN^\xi)^{\frac{d}{\alpha}}\}$ under the condition $M_N\leq 2^{k+1}BN^\xi$.
\item[(2)] The entropy cost arises from two sources: the sub-diffusive or super-diffusive behavior of the random walk, and the atypical configuration of its range. We analyze these together with the range penalty.
\begin{itemize}\itemsep -1pt
\item[(2a)] Note that $\xi<\frac12<1-\frac{\zeta}{2}$. Thus, the random walk is sub-diffusive for small $k$ such that $2^k BN^\xi\ll\sqrt{N}$, and super-diffusive for large $k$ such that $2^k BN^{\xi}\gg\sqrt{N}$.
\item[(2b)] Concerning the range of the walk, a so-called \textit{droplet} phenomenon may occur (see \cite{Bolt94}): the walk occasionally visits a site atypically far away, and then return to its main folded region. The length of the droplet plays a crucial role. If it is much shorter than $N^{d\xi}$, then the range penalty is still of order $N^{d\xi-\zeta}$. However, if it is much longer, for instance of order $N^{\vartheta}$ with $\vartheta>d\xi$, then the range penalty should be of order $N^{\vartheta\wedge1-\zeta}$. Furthermore, in the latter case, it is unclear whether the range consists of a main folded region together with a single droplet, or of multiple connected folded regions. In these two scenarios, the entropy cost is expected to be significantly different.
\end{itemize} 
\end{itemize}

To compare the energy gain, the range penalty, and the entropy cost analyzed above, the main challenge lies in estimating the entropy cost arising from atypical paths of the random walk. One possible approach is to decouple the range penalty and the energy gain by H\"{o}lder's inequality as in \eqref{eq:Holder}, and then apply large deviation estimates for the events $\{M_N\geq 2^k BN^\xi\}$ or the range $\cR_N$ with droplets under the homogeneous polymer measure. However, the existing results (see \cite{Szn95a,Szn95b,BBH01,F08,DFSX20+}) appear insufficient for our purposes. In particular, we require uniform and sufficiently sharp estimates depending on $k$ and $B$ in order to show that \eqref{eq:term_extra} is negligible. At present, we do not see a clear approach to tackle this problem.

Fundamental difficulties also arise in part (a), corresponding to the term $\bZ_N(M_N\leq BN^\xi)$. To show that the energy gain is negligible, we may apply H\"{o}lder's inequality as in \eqref{eq:Holder} and \eqref{eq:invers_Holder} to get
\begin{align}
\label{eq:discuss1}\bZ_N(M_N\leq B N^\xi)&\geq \bZ_{N,\beta_N=0}^{\omega,(1-\epsilon)h_N}(M_N\leq B N^\xi)^{\frac{1}{1-\epsilon}}\bZ_{N,-\frac{1-\epsilon}{\epsilon}\beta_N}^{\omega,h_N=0}(M_N\leq B N^\xi)^{-\frac{\epsilon}{1-\epsilon}},\\
\label{eq:discuss2}\bZ_N(M_N\leq B N^\xi)&\leq \bZ_{N,\beta_N=0}^{\omega,(1+\epsilon)h_N}(M_N\leq B N^\xi)^{\frac{1}{1+\epsilon}}\bZ_{N,\frac{1+\epsilon}{\epsilon}\beta_N}^{\omega,h_N=0}(M_N\leq B N^\xi)^{\frac{\epsilon}{1+\epsilon}}.
\end{align}
We shall show that the second factors on the right-hand side of \eqref{eq:discuss1} and \eqref{eq:discuss2} are of order $N^{\frac{d\xi}{\alpha}-\gamma}$. However, the strategy in \cite[Lemma 6.2]{BTW21} does not apply here, since the random walk therein is supper-diffusive, whereas in our case $\xi<\frac{1}{2}$. Hence, a sharp estimate of the entropy cost arising from sub-diffusivity is required.

We note that some special cases can be treated. In particular, by the weak symmetrization inequality (see \cite{Gut05}) and the Fuk–Nagaev inequality \cite{N79}, one can show that
\begin{equation*}
\bbP\Big(\sum\limits_{\|x\|\leq\ell}|\omega_x|>T\Big)\leq C\ell^d T^{-(\alpha\wedge2)}\quad\text{(a logarithmic correction is needed for $\alpha=2$)}.
\end{equation*}
Then for $\alpha<2$, by \eqref{eq:R5b}, a very rough bound yields
\begin{equation*}
\begin{split}
\bbP\Big(\bZ_{N,\delta\beta_N}^{\omega,h_N=0}\big(M_N\leq BN^\xi\big)>e^{\epsilon N^{1-2\xi}}\Big)&\leq\bbP\Big(\delta\beta_N\sum\limits_{\|x\|\leq BN^\xi}|\omega_x|>\epsilon N^{1-2\xi}\Big)\leq C_\epsilon N^{d\xi-\alpha(1+\gamma-2\xi)}\\
&=C N^{-\alpha(\gamma-\frac{d+2\alpha}{\alpha(2+d)}\zeta-\frac{d(1-\alpha)}{\alpha(2+d)})}\overset{N\to\infty}{\longrightarrow}0,
\end{split}
\end{equation*}
but this does not work for $\alpha>2$.

From the above analysis, one can see that a key step is to show that $\bZ_N(M_N > B N^\xi)$ is negligible. As a byproduct, this would also provide a rigorous characterization of the end-to-end exponent $\xi$ given in Definition~\ref{def:xi}.

In summary, if one can choose a suitable probability measure (either the original measure $\bP$ or the polymer measure $\bP_N$) and establish corresponding deviation estimates for $M_N$ and $|\cR_N|$, then it appears promising to resolve the remaining part of the region $R_5$. Such results would simultaneously allow us to identify the end-to-end fluctuations in the sense of Definition~\ref{def:xi}, as in \cite{BTW21}.

\section*{Acknowledgments} N.T.\ is supported by the ANR Local (ANR-22-CE40-0012-02) and the project Labex MME-DII (ANR11-LBX-0023-01). R. W.\ is supported by National Natural Science Foundation of China grant 12401170 and Xi’an Jiaotong-Liverpool University Research Development Fund RDF-23-01-024..

We thank Quentin Berger and Nicolas Bouchot for the helpful discussions, especially in solving the region $R_4$. We also thank Amine Asselah, Ryoki Fukushima and Bruno Schapira for the comments, suggestions and discussions, which provide invaluable insight to the problem for us. Finally, we are grateful to the referees for their constructive suggestions and comments, which have significantly improved the paper.

\appendix

\section{Control on the last passage percolation with a given range}\label{sec:appA}
Recall the notation $M_N$, $(\bM_i^{(r)},\bY_i^{(r)})$ and $\Omega_r(\Delta)$ from Section \ref{S:notations}. For $m\in\bbN$, $p>1$ and $s\geq1$, let 
\begin{equation}\label{eq:def_L}
\mathsf L_m(r,s)=\mathsf L_m^p(r,s):=\max\Big\{|\Delta|\colon \Delta= \Upsilon_m^{(pr)}\cap \cR_N(S), S\in \cM(r,s)\Big\},
\end{equation}
where $ \Upsilon_m^{(r)}:=\{\bY_i^{(r)}, 1\leq i\leq m\}$ is the set of $m$ different sites taken uniformly on $\Lambda_r$ and
\begin{equation}\label{eq:Mrs}
\cM(r,s)=\cM_p(r,s):=\Big\{S\colon \,  M_N(S)\in [r,pr],\, |\cR_N(S)|\in [rs ,p^2 rs]\Big\},
\end{equation}

\begin{remark}\label{rem:ordDelta}
Let us observe that in the definition of $\mathsf L_m(r,s)$ we only consider the cardinality of a finite set $\Delta$. In the sequel, it is convenient to look at $\Delta$ as an ordered set. For this purpose, we pick at random a trajectory $S \in \cM(r,s)$ with $\Delta=\Upsilon_m^{(pr)}\cap \cR_N(S)$, so that any set $\Delta$ is associated with an unique random walk trajectory $S$.
Therefore, if $\Delta=\{x_1, \ldots, x_n\}$, then we order the points $x_1, \ldots, x_n$ according to $S$, that is, $x_1$ is the first point visited by $S$, $x_2$ the second one and so on. 
We stress the fact that the only important point is that $\Delta$ is ordered according to a random walk trajectory, in order to control its entropy, see \eqref{def:ent} and \eqref{eq:revL}.
\end{remark}

The following result gives us a control on $\mathsf L_m(r,s)$.
\begin{theorem}\label{thm:transf-rangentr}
For any fixed $p>1$, there exists a constant $c>0$ which depends only on $p$ and the dimension $d$, such that for any $m\in\bbN$, $r,s\geq1$ and $1\le k\le |\Lambda_{pr}|$,
\begin{equation}
\bbP\Big(\mathsf L_m(r,s)\ge k\Big) \le \Bigg(\frac{c\, s\, m^{1/d}}{k}\Bigg)^{dk}.
\end{equation}
In particular,  for any $b\in[0,d]$ there exists a constant $c=c(d,b, p)$ such that for any $m\in\bbN$ and $r,s>0$, 
\begin{equation}\label{eq:markovforL}
\bbE\Bigg[\Bigg(\frac{\mathsf L_m(r,s)}{(s m^{1/d})\wedge m}\Bigg)^b\Bigg]\le c.
\end{equation}
\end{theorem}
\begin{proof} 
We first introduce an \textit{entropy} for any ordered $(x_1,\cdots,x_n)$ by
\begin{equation}\label{def:ent}
\text{Ent}((x_1,\cdots,x_n)):=\frac{d}{2}\Big(\sum\limits_{k=1}^n\|x_k-x_{k-1}\|\Big)^2
\end{equation}
with the convention $x_0$ being the origin (see \cite{BTW21}). According to Remark \ref{rem:ordDelta}, for any $\Delta$ we have associated an unique random walk trajectory $S$ such that $\Delta \subset\cR_N(S)$ and the points of $\Delta$ are ordered according to $S$. Then, we can write $\Delta=(x_1,\cdots,x_n)$, and we have $\sum_{i=1}^n\|x_i-x_{i-1}\|\le C_d|\cR_N(S)|\leq C_d p^2 rs$ by the configuration of the simple random walk. Hence, by setting $B=\frac{d}{2}C_d^2p^4 (rs)^2$,
\begin{equation}\label{eq:revL}
\begin{split}
\mathsf L_m(r,s)&=\max\Big\{|\Delta|\colon \Delta= \Upsilon_m^{(pr)}\cap \cR_N(S), S\in \cM(r,s), \text{Ent}(\Delta)\le B\Big\}\\
&\leq L_m^{(B)}=L_{m,p,r}^{(B)}:=\max\Big\{|\Delta|\colon \Delta\subset \Upsilon_m^{(pr)}, \, \text{Ent}(\Delta)\le B\Big\}.
\end{split}
\end{equation}
Finally,  by \cite[Theorem A.1]{BTW21}, we have that
$$
\bbP\Big(\mathsf L_m(r,s)\ge k\Big)\le \bbP\Big( L_m^{(B)} \ge k\Big)\le \Big(\frac{c B^{1/2}m^{1/d}}{rk}\Big)^{dk},
$$
and \eqref{eq:markovforL} follows by \cite[Corollary A.2]{BTW21}.

%
\end{proof}

Then we have the following control on the energy gain associated with the range $\cR_N$, which is analogous to \cite[Lemma A.3]{BTW21}.
\begin{proposition}\label{lem:Omega}
Let $\cM(r,s)$ be as in \eqref{eq:Mrs}, then there exists a constant $C>0$ such that for any $t>0$, $p>1$ and $s\geq1$,
\begin{equation}\label{eq:lemOmega}
\bbP\Bigg(\max\limits_{S\in \cM(r,s) }\Omega_{r}(\cR_N(S))>r^{\frac{d}{\alpha}}st\Bigg)\\
\le C t^{-\frac{\alpha d}{\alpha+d}}.
\end{equation}
\end{proposition}
\begin{proof}
For $S\in \cM(r,s)$ and $p>1$, we have that (for simplicity, we omit $\lfloor\cdot\rfloor$),
\begin{equation*}
\Omega_{r}(\cR_N(S))=\sum\limits_{i=1}^{\log_p|\Lambda_r|}\sum\limits_{j=p^{i-1}}^{p^i-1}\bM_j^{(r)}\ind_{\{\bY_j^{(r)}\in\cR_N(S)\}}\le \sum_{i=1}^{\log_p|\Lambda_{r}|} \bM_{p^{i-1}}^{(r)}\mathsf L_{p^i}(r,s),
\end{equation*}
where we use that $\bM_{p^{i-1}}^{(r)}\geq \bM_j^{(r)}$ for $j\geq p^{i-1}$ and $\sum_{j\leq p^i}\ind_{\{\bY_j^{(r)}\in\cR_N(S)\}}=|\Upsilon_{p^i}^{(r)}\cap\cR_N|\leq\mathsf L_{p^i}(r,s)$.

Then the probability in \eqref{eq:lemOmega} is bounded by
\begin{equation*}
\sum_{i=1}^{\log_p|\Lambda_r|}\bbP\Bigg(\bM_{p^{i-1}}^{(r)}\mathsf L_{p^i}(r,s)>C p^{-i(\frac1\alpha-\frac1d-2\delta)}r^{\frac{d}{\alpha}}st\Bigg),
\end{equation*}
for some sufficiently small $\delta>0$ since $\alpha<d$ with $C=\big(\sum_{i=1}^{\infty}p^{-i(\frac1\alpha-\frac1d-2\delta)}\big)^{-1}$.

We then bound the above probability by
\begin{equation}
\sum_{i=1}^{\log_p|\Lambda_r|}\Bigg(\bbP\Big(\bM_{p^{i-1}}^{(r)}> Ct^{1-\eta}p^{-\frac{i}\alpha+i\delta}r^{\frac{d}{\alpha}}\Big) + 
\bbP\Big(
\mathsf L_{p^i}(r,s)>t^\eta p^{\frac{i}d+i\delta}s\Big)\Bigg).
\end{equation}
Since $\mathbb P(\bM_\ell^{|\Lambda_r|}>Q)\le (\ell^{\frac1\alpha} r^{-\frac{d}\alpha}Q)^{-\alpha \ell}$ by standard large deviation estimates of heavy-tailed random variables (see \cite[Lemma 5.1]{BT19}), the first sum is smaller than 
$$
C\sum_{i=1}^{\log_p|\Lambda_r|}p^{-i\delta \alpha p^i}t^{-\alpha p^i(1-\eta)}\le C t^{-\alpha(1-\eta)}.
$$
The second sum is controlled by Markov's inequality and \eqref{eq:markovforL} with $b=d$, so that it is smaller than
$$
\sum_{i=1}^{\log_p|\Lambda_r|}t^{-d\eta}p^{-id\delta}\le C t^{-d\eta}.
$$
In such a way \eqref{eq:lemOmega} is bounded by $C( t^{-\alpha(1-\eta)} +t^{-d\eta})$, which is optimized by taking $\eta=\frac\alpha{\alpha+d}$.
\end{proof}

\section{Extending known results in the regions  $R_1$, $R_2$ and $R_3$}\label{S-R123}
We now explain how to adapt the approaches in \cite{BTW21} to prove the results in the regions $R_1$, $R_2$ and $R_3$.

\subsection{Region $R_1$}
We start by Theorem \ref{thm:R1b} and the case (ii) of Theorem \ref{thm:R1c}. First note that by the same application of H\"{o}lder's inequality as \eqref{eq:Holder} and \eqref{eq:invers_Holder}, we have that for any $\epsilon>0$,
\begin{equation}\label{eq:range}
\Big(\bZ_{N,\beta_N=0 }^{\omega, (1-\epsilon)h_N}\Big)^{\frac{1}{1-\epsilon}}\Big(\bZ_{N, -\frac{1-\epsilon}{\epsilon}\beta_N}^{\omega, h_N=0}\Big)^{-\frac{\epsilon}{1-\epsilon}}\leq\bZ_N\leq
\Big(\bZ_{N, \beta_N=0}^{\omega, (1+\epsilon)h_N}\Big)^{\frac{1}{1+\epsilon}}
\Big(\bZ_{N,\frac{1+\epsilon}{\epsilon}\beta_N}^{\omega, h_N=0}\Big)^{\frac{\epsilon}{1+\epsilon}},
\end{equation}
and
\begin{equation}\label{eq:energy}
\Big(\bZ_{N,(1-\epsilon)\beta_N}^{\omega, h_N=0}\Big)^{\frac{1}{1-\epsilon}}\Big(\bZ_{N, \beta_N=0}^{\omega, -\frac{1-\epsilon}{\epsilon}h_N}\Big)^{-\frac{\epsilon}{1-\epsilon}}\leq\bZ_N\leq
\Big(\bZ_{N, (1+\epsilon)\beta_N}^{\omega, h_N=0}\Big)^{\frac{1}{1+\epsilon}}
\Big(\bZ_{N,\beta_N=0}^{\omega, \frac{1+\epsilon}{\epsilon}h_N}\Big)^{\frac{\epsilon}{1+\epsilon}}.
\end{equation}In view of equations (2.8) and (2.11) in \cite{BTW21}, we know that the proper scaling for $\log\bZ_{N,\beta_N}^{\omega,h_N=0}$ is $v_N/(\beta_N N^{d/2\alpha})$. On the other hand, by $\zeta>1$,
\begin{equation*}
\bZ_{N,\beta_N=0}^{\omega,h_N}-1=\bE\Big[e^{-h_N|\cR_N|}-1\Big]\sim -h_N\bE[|\cR_N|],
\end{equation*}
which is the proper scaling for $\log\bZ_{N,\beta_N=0}^{\omega,h_N}$. We then compare $1/(h_N\bE[|\cR_N])$ and $v_N/(\beta_N N^{\frac{d}{2\alpha}})$, and we have that
\begin{equation}\label{eq:compare}
v_N\frac{h_N\bE[|\cR_N|]}{\beta_N N^{\frac{d}{2\ga}}}\sim CN^{\gamma-\zeta+\frac{d(\alpha-1)}{2\ga}},
\end{equation}
which provides the threshold $\zeta=\gamma+\frac{d(\alpha-1)}{2\alpha}$ in Theorem \ref{thm:R1b} and in the case (ii) of Theorem \ref{thm:R1c}. The results then follow by taking logarithms in \eqref{eq:range} and \eqref{eq:energy}, sending $N\to\infty$, and finally sending $\epsilon \to 0$.


\vspace{0.2cm}
Then we turn to Theorem \ref{thm:R1a} and the case (i) of Theorem \ref{thm:R1c}. The treatment \eqref{eq:range} and \eqref{eq:energy} does not work well here, due to the non-trivial centering. We follow the approach in \cite[Section 7.1]{BTW21}.

First, we truncate the random environment by $\tilde{\omega}:=\omega\ind_{\{|\omega|\leq k_N\}}$ with $k_N=N^{\gamma-\epsilon}$ and $\epsilon>0$ small enough. Let $\lambda_N:=\log\bbE[\exp(\beta_N\tilde{\omega})]$. Then we have that
\begin{equation}\label{eq:lambda}
|\lambda_N|\sim\beta_N k_N^{1-\alpha}(\log k_N)^{\ind_{\{\alpha=1,2\}}}~\text{for}~\alpha\in(0,2]\quad\text{and}\quad|\lambda_N|=\frac12\beta_N^2+O(\beta_N k_N^{1-\alpha}\log k_N)~\text{for}~\alpha\in(2,d),
\end{equation}
by $k_N\beta_N\overset{N\to\infty}{\longrightarrow}0$, Taylor expansion and the assumption \eqref{def:omega}.

Since $\bZ_{N,\beta_N}^{\omega,h_N}\leq\bZ_{N,\beta_N}^{\omega,h_N=0}$ and $\bZ_{N,\beta_N}^{\tilde{\omega},h_N}\leq\bZ_{N,\beta_N}^{\tilde{\omega},h_N=0}$, by the method in \cite[Lemma 7.1]{BTW21}, we can work on $\bZ_{N,\beta_N}^{\tilde{\omega},h_N}$. Perform a polynomial chaos expansion and we have that (also see equation (7.11) in \cite{BTW21})
\begin{equation}\label{eq:poly}
\bZ_{N,\beta_N}^{\tilde{\omega},h_N}=\bE\Big[e^{(\lambda_N-h_N)|\cR_N|}\Big]+\sum\limits_{x\in\bbZ^d}\epsilon_x\bE\Big[e^{(\lambda_N-h_N)|\cR_N|}\ind_{\{x\in\cR_N\}}\Big]+\bY_N,
\end{equation}
where $\epsilon_x=\exp(\beta_N\tilde{\omega}_x-\lambda_N)-1$ with $\bbE[\epsilon_x]=0$ and $\bbE[(\epsilon_x)^2]\leq C(\beta_N^2\bbE[(\tilde{\omega}_x)^2]+\lambda_N^2)$, in particular,
\begin{equation}\label{eq:e_x}
\bbE\big[(\epsilon_x)^2\big]\leq C\beta_N k_N^{1-\alpha}~\text{for}~\alpha\in(0,2)\quad\text{and}\quad\bbE\big[(\epsilon_x)^2\big]\leq C\beta_N^2~\text{for}~\alpha\in(2,d),
\end{equation}
and $\bY_N$ is the higher order term, given by
\begin{equation}
\bY_N:=\sum\limits_{k\geq2}\sum\limits_{(x_1,\cdots,x_k)\in(\bbZ^d)^k\atop x_i\neq x_j, i,j=1,\cdots,k}\Big(\prod\limits_{i=1}^k\epsilon_{x_i}\Big)\bE\Big[e^{(\lambda_N-h_N)|\cR_N|}\ind_{\{(x_1,\cdots,x_k)\in\cR_N\}}\Big].
\end{equation}
As equation (7.14) in \cite{BTW21}, and noting that $h_N|\cR_N|=N^{1-\zeta}\overset{N\to\infty}{\longrightarrow}0$, we still have
\begin{equation*}
\lim\limits_{N\to\infty}\frac{1}{a_N^2\beta_N^2}\bbE\big[\big(\bY_N\big)^2\big]=0.
\end{equation*}


Then we work on \eqref{eq:a>2} in Theorem \ref{thm:R1a}. For the centering term, note that, since $\alpha > 2$ and by \eqref{eq:lambda}, a comparison with the first term in \eqref{eq:poly} yields
\begin{equation}\label{eq:second_term}
\begin{split}
&\frac{\bE\Big[\Big|e^{(\lambda_N-h_N)|\cR_N|}-e^{(\frac{1}{2}\beta_N^2-h_N)|\cR_N|}\Big|\Big]}{a_N\beta_N}\leq C\frac{|\lambda_N-\frac{1}{2}\beta_N^2|N}{a_N\beta_N}
\leq C\frac{N k_N^{1-\alpha}\log k_N}{a_N}\\
\ll&\frac{\beta_N^{\alpha-1}N^{1+\delta}}{a_N}=\frac{N^{-(\alpha-1)\gamma+1+\delta}}{a_N},\quad\text{where}~a_N:=\begin{cases}
N^{\frac14}, &\quad\text{if}~d=3,\\
\sqrt{\log N}, &\quad\text{if}~d=4,\\
1, &\quad\text{if}~d\geq5,
\end{cases}
\end{split}
\end{equation}
and $\delta>0$ can be chosen sufficiently small. The above quantity tends to $0$ as $N\to\infty$ since $\gamma>\frac{d}{2\alpha}$ and $\alpha>1$. In particular, by writing $\gamma=\frac{d}{2\alpha}+(\alpha-1)\delta'$ with $\delta'>0$, we have that on the exponent of the numerator, $-(\alpha-1)\gamma+1=-\frac{d}{2}+\frac{d}{2\alpha}+1-\delta'<1-\frac{d}{2}+\frac{d}{4}-\delta'=1-\frac{d}{2}-\delta'$. It suffices to choose $\delta<\delta'$.

For the second term in \eqref{eq:poly}, we rewrite it as
\begin{equation}\label{eq:two_terms}
\sum\limits_{x\in\bbZ^d}\epsilon_x\bE\Big[\big(e^{(\lambda_N-h_N)|\cR_N|}-1)\ind_{\{x\in\cR_N\}}\Big]+\sum\limits_{x\in\bbZ^d}\epsilon_x\bP(x\in\cR_N).
\end{equation}
When $\alpha>2$, $|\lambda_N|\sim\frac12\beta_N^2$. If $\zeta\geq2\gamma$, then $h_N=O(\lambda_N)$, and the case is identical to equation (7.16) in \cite{BTW21}. If $\zeta<2\gamma$, then $\lambda_N=o(h_N)$, and for the first term of \eqref{eq:two_terms}, we have that by $\zeta>1$, its second moment is bounded above by
\begin{equation*}
\frac{\beta_N^2 h_N^2 N^2\bE[|\cR_N\cap\cR'_N|]}{a_N^2\beta_N^2}\leq \frac{N^{2-2\zeta}}{a_N^2}\times\begin{cases}
\sqrt{N}, &\quad\text{if}~d=3,\\
\log N, &\quad\text{if}~d=4,\\
c_d, &\quad\text{if}~d\geq5,
\end{cases}
\quad\overset{N\to\infty}{\longrightarrow}0.
\end{equation*}
The convergence of the second term in \eqref{eq:two_terms} to $\cX$ was established in \cite[Section~7.2]{BTW21}.

The treatment of \eqref{eq:old1} in Theorem~\ref{thm:R1c} is analogous to that of \eqref{eq:a>2} above, and we only need to adjust the centering term. Noting that for $\alpha>\frac{d}{d-2}>1$, we have that
\begin{equation*}
\frac{\bE\Big[\Big|e^{(\lambda_N-h_N)|\cR_N|}-e^{-h_N|\cR_N|}\Big|\Big]}{\beta_N}\leq\frac{|\lambda_N|N}{\beta_N}\leq N^{-(\alpha-1)\gamma+1+\delta}\leq N^{-\frac{d}{2}+\frac{d}{2\alpha}+1-\delta'}\leq N^{-\delta'}\overset{N\to\infty}{\longrightarrow}0,
\end{equation*}
where we used that $\gamma > \frac{d}{2\alpha}$. The treatment of all other terms in \eqref{eq:poly} is identical to that for \eqref{eq:a>2}.

\vspace{0.2cm}
The end-to-end fluctuations in the region $R_1$ can be shown by the same method in \cite[Section 7.3]{BTW21}.

\subsection{Regions $R_2$ and $R_3$}

By the trivial bound $-|\cR_N|\geq-N$, we have that
\begin{equation*}
\bZ_N^{h_N=0}e^{-\hat{h}N^{1-\zeta}}\leq\bZ_{N}\leq\bZ_N^{h_N=0}.
\end{equation*}
Then in the region $R_2$, by taking logarithm of $e^{\hat{h}N^{1-\zeta}}$ and dividing it by $N^{2\xi-1}$, we have that
\begin{equation*}
N^{1-\zeta-(2\xi-1)}=N^{2-\zeta-\frac{2\ga(1-\gamma)}{2\ga-d}}=N^{\frac{2\ga}{2\ga-d}(\gamma-\frac{2\ga-d}{2\ga}\zeta-\frac{d-\ga}{\ga})}\overset{N\to\infty}{\longrightarrow}0.
\end{equation*}
Similarly, in the region $R_3$, we have that
\begin{equation*}
\frac{N^{1-\zeta}}{\beta_N N^{\frac{d}{\ga}}}=N^{1-\zeta+\gamma-\frac{d}{\ga}}=N^{\gamma-\zeta-\frac{d-\ga}{\ga}}\overset{N\to\infty}{\longrightarrow}0.
\end{equation*}

Hence, in the regions $R_2$ and $R_3$, the range penalty is negligible, and the approaches in \cite{BTW21} can be adapted here without changes. The end-to-end fluctuations in the regions $R_2$ and $R_3$ can also be shown by the same method in \cite[Section 5.4, Section 6.2]{BTW21}.

\bibliographystyle{plain}
\bibliography{references}

\begin{thebibliography}{10}

\bibitem{AKQ14}
T.~Alberts, K.~Khanin, and J.~Quastel.
\newblock The intermediate disorder regime for directed polymers in dimension 1
  + 1.
\newblock {\em Ann. Probab.}, 42:1212--1256, 2014.

\bibitem{BW03}
Itai Benjamini and David Wilson.
\newblock Excited random walk.
\newblock {\em Electron. Commun. Probab.}, 8:86--92, 2003.

\bibitem{BC21}
N.~Berestycki and R.~Cerf.
\newblock The random walk penalised by its range in dimensions $d\geq3$.
\newblock {\em Annales Henri Lebesgue}, 4:1--79, 2021.

\bibitem{BHTW20}
Q.~Berger, C.-H. Huang, N.~Torri, and R.~Wei.
\newblock One-dimensional polymers in random environments: stretching vs.
  folding.
\newblock {\em Electronic Journal of Probability}, 27:1--45, 2022.

\bibitem{BT19}
Q.~Berger and N.~Torri.
\newblock Entropy-controlled last-passage percolation.
\newblock {\em The Annals of Applied Probability}, 29(3):1878--1903, 2019.

\bibitem{BT21}
Q.~Berger and N.~Torri.
\newblock Beyond hammersley’s last-passage percolation: a discussion on
  possible local and global constraints.
\newblock {\em Annales de l’Institut Henri Poincar{\'e} D}, 8(2):213--241,
  2021.

\bibitem{BTW21}
Q.~Berger, N.~Torri, and R.~Wei.
\newblock Non-directed polymers in heavy-tail random environment in dimension
  $d\geq2$.
\newblock {\em Electronic Journal of Probability}, 27:1--67, 2022.

\bibitem{BT19a}
Quentin Berger and Niccol{\`o} Torri.
\newblock Directed polymers in heavy-tail random environment.
\newblock {\em Ann. Probab.}, 47(6):4024--4076, 2019.

\bibitem{Bolt94}
E.~Bolthausen.
\newblock Localization of a two-dimensional random walk with an attractive path
  interaction.
\newblock {\em Ann. Probab.}, pages 875--918, 1994.

\bibitem{B24}
N.~Bouchot.
\newblock Scaling limit of a one-dimensional polymer in a repulsive iid
  environment.
\newblock {\em Electronic Journal of Probability}, 29:1--43, 2024.

\bibitem{B24+}
N.~Bouchot.
\newblock Scaling limits for the random walk penalized by its range in
  dimension one.
\newblock {\em ALEA. Latin American Journal of Probability \& Mathematical
  Statistics}, 21(1), 2024.

\bibitem{CSZ13}
F.~Caravenna, R.~Sun, and N.~Zygouras.
\newblock Polynomial chaos and scaling limits of disordered systems.
\newblock {\em J. Eur. Math. Soc. (JEMS)}, 19:1--65, 2017.

\bibitem{C17}
F.~Comets.
\newblock {\em Directed Polymers in Random Environments: \'{E}cole
  d'\'{E}t\'{e} de Probabilit\'{e}s de Saint-Flour XLVI ---2016}, volume 2175.
\newblock Springer, 2017.

\bibitem{DFSX20+}
J.~Ding, R.~Fukushima, R.~Sun, and C.~Xu.
\newblock Biased random walk conditioned on survival among bernoulli obstacles:
  subcritical phase.
\newblock {\em Communications in Mathematical Physics}, 376(3):2161--2195,
  2020.

\bibitem{DFSX20}
J.~Ding, R.~Fukushima, R.~Sun, and C.~Xu.
\newblock Geometry of the random walk range conditioned on survival among
  bernoulli obstacles.
\newblock {\em Probability Theory and Related Fields}, 177:91--145, 2020.

\bibitem{DFSX21}
J.~Ding, R.~Fukushima, R.~Sun, and C.~Xu.
\newblock Distribution of the random walk conditioned on survival among
  quenched bernoulli obstacles.
\newblock {\em Ann. Probab.}, 49(1), 2021.

\bibitem{DV1979}
M.~D. Donsker and S.~R.~S. Varadhan.
\newblock On the number of distinct sites visited by a random walk.
\newblock {\em Communications on Pure and Applied Mathematics}, 32:721--747,
  1979.

\bibitem{DE51}
A.~Dvoretzky and P.~Erd{\"o}s.
\newblock Some problems on random walk in space.
\newblock {\em Proceedings of the Second Berkeley Symposium on Mathematical
  Statistics and Probability}, page 353–367, 1951.

\bibitem{F08}
R.~Fukushima.
\newblock Asymptotics for the wiener sausage among poissonian obstacles.
\newblock {\em Journal of Statistical Physics}, 133:639--657, 2008.

\bibitem{Gut05}
A.~Gut.
\newblock {\em Probability: A Graduate Course}.
\newblock 2005.

\bibitem{H19}
Chien-Hao Huang.
\newblock The scaling limits for wiener sausages in random environments.
\newblock {\em arXiv preprint arXiv:1902.04930}, 2019.

\bibitem{N79}
S.~V. Nagaev.
\newblock Large deviations of sums of independent random variables.
\newblock {\em The Annals of Probability}, 7(5):745--789, 1979.

\bibitem{S90}
U.~Schmock.
\newblock Convergence of the normalized one-dimensional wiener sausage path
  measures to a mixture of brownian taboo processes.
\newblock {\em Stochastics: An International Journal of Probability and
  Stochastic Processes}, 29(2):171--183, 1990.

\bibitem{Szn95a}
Alain-Sol Sznitman.
\newblock Annealed lyapounov exponents and large deviations in a poissonian
  potential. i.
\newblock In {\em Annales scientifiques de l'Ecole normale sup{\'e}rieure},
  volume~28, pages 345--370, 1995.

\bibitem{Szn95b}
Alain-Sol Sznitman.
\newblock Annealed lyapounov exponents and large deviations in a poissonian
  potential. ii.
\newblock In {\em Annales scientifiques de l'Ecole normale sup{\'e}rieure},
  volume~28, pages 371--390, 1995.

\bibitem{BBH01}
M.~van~den Berg, E.~Bolthausen, and F.~den Hollander.
\newblock Moderate deviations for the volume of the wiener sausage.
\newblock {\em Annals of Mathematics}, pages 355--406, 2001.

\bibitem{Zer05}
Martin~PW Zerner.
\newblock Multi-excited random walks on integers.
\newblock {\em Probab. Theory Relat. Fields}, 133(1):98--122, 2005.

\bibitem{Z24}
Nikos Zygouras.
\newblock Directed polymers in a random environment: a review of the phase
  transitions.
\newblock {\em Stochastic Processes and their Applications}, 177:104431, 2024.

\end{thebibliography}

\end{document}